\documentclass[11pt, notitlepage]{article}
\usepackage{fullpage}

\usepackage{amsfonts,amsmath,amsthm,amssymb,fancyhdr,float,caption,mathtools,enumitem,nicefrac,verbatim,graphicx,tikz-cd,wasysym}
\usepackage[pdfencoding=auto, unicode,breaklinks=true,colorlinks=true]{hyperref} 

\theoremstyle{plain}
\newtheorem{theorem}{Theorem}[section]
\newtheorem{prop}[theorem]{Proposition}

\newtheorem{lemma}[theorem]{Lemma}
\newtheorem{cor}[theorem]{Corollary}

\theoremstyle{definition}
\newtheorem{definition}[theorem]{Definition}

\newcommand{\La}{\Lambda} 
\newcommand{\la}{\lambda} 
\newcommand{\al}{\alpha} 
\newcommand{\be}{\beta}
\newcommand{\Om}{\Omega} 
\newcommand{\te}{\theta} 

\newcommand{\Ga}{\Gamma}
\newcommand{\om}{\omega} 
\newcommand{\Si}{\Sigma} 
 
\newcommand{\ep}{\epsilon}
\newcommand{\varep}{\varepsilon}
\newcommand{\De}{\Delta}
\newcommand{\si}{\sigma}
\newcommand{\io}{\iota}
\newcommand{\de}{\delta}

\newcommand{\R}{{\mathbb R}}
\newcommand{\Z}{{\mathbb Z}}

\newcommand{\T}{{\mathbb T}}

\newcommand{\abs}[1]{\left|#1\right|}
\newcommand{\deriv}[2]{\frac{\partial #1}{\partial #2}}
\renewcommand{\C}{{\mathbb C}}
\renewcommand{\geq}{\geqslant}
\renewcommand{\leq}{\leqslant}

\newcommand{\tex}[1]{\texorpdfstring{#1}{?}}
\newenvironment{remark}{\refstepcounter{theorem}\par\medskip\noindent{\bf
Remark~\thetheorem.}}{\unskip\nobreak\hfill\hbox{ $\oslash$}\par\bigskip}

\newenvironment{example}{\refstepcounter{theorem}\par\medskip\noindent{\bf
Example~\thetheorem.}}{\unskip\nobreak\hfill\hbox{ $\oslash$}\par\bigskip}

\newcommand{\imm}{\hookrightarrow}
\newcommand{\adjustedarrow}[1]{\xhookrightarrow{\raisebox{-1.5pt}[3pt][0pt]{\ensuremath{\scriptstyle{#1}}}\,}}

\newcommand{\immrk}{\adjustedarrow{\R^k}}
\newcommand{\immrn}{\adjustedarrow{\R^n}}

\newcommand{\immtn}{\adjustedarrow{\T^n}}

\newcommand{\immG}{\adjustedarrow{G}}
\newcommand{\immTR}{\adjustedarrow{S^1\times\R}}

\newcommand{\cB}[2]{c_\mathrm{B}^{#1,#2}}

\newcommand{\cBmk}{\cB{m}{k}}
\newcommand{\cBoo}{\cB{1}{1}}
\newcommand{\cBnn}{\cB{n}{n}}
\newcommand{\cgromov}{c_{\mathrm{B}}}

\newcommand{\toricpack}{\mathcal{T}}
\newcommand{\semitoricpack}{\mathcal{ST}}

\newcommand{\sympplain}{\mathrm{Symp}}
\newcommand{\symp}{\sympplain^{2n}}
\newcommand{\sympG}{\sympplain^{2n,G}}
\newcommand{\sympT}{\sympplain^{2n,\T^n}_{\mathrm{T}}}
\newcommand{\sympST}{\sympplain^{4,S^1\times\R}_{\mathrm{ST}}}
\newcommand{\hamG}{\mathrm{Ham}^{2n, G}}
\newcommand{\hamT}{\mathrm{Ham}^{2n,\T^n}_{\mathrm{T}}}
\newcommand{\hamST}{\mathrm{Ham}^{4,S^1\times\R}_{\mathrm{ST}}}

\newcommand{\msympT}{\sympT/{\sim}_{\mathrm{T}}}
\newcommand{\msympST}{\sympST/{\sim}_{\mathrm{ST}}}
\newcommand{\mhamT}{\hamT/{\approx}_{\mathrm{T}}}
\newcommand{\mhamST}{\hamST/{\approx}_{\mathrm{ST}}}

\newcommand{\symdiff}{\ast}
\newcommand{\vol}{\mathrm{vol}}
\newcommand{\pack}{\pi}
\renewcommand{\P}{\mathcal{P}}
\newcommand{\semitoric}{\mathrm{ST}}
\newcommand{\toric}{{\mathrm{T}}}
\newcommand{\vertr}{\mathrm{Vert}(\R^2)}
\newcommand{\vect}[2]{\begin{pmatrix}#1\\#2\end{pmatrix}}

\newcommand{\Rxy}{\R[[X,Y]]}
\newcommand{\Rxyz}{\R[[X,Y]]_0}
\newcommand{\mf}{{m_f}}
\newcommand{\polyg}{\mathrm{Polyg}(\R^2)}
\newcommand{\lwpolyg}{\mathcal{LW}\polyg}
\newcommand{\stpolyg}{\mathrm{Polyg}_{\semitoric}(\R^2)}
\newcommand{\stpolygn}{\mathrm{Polyg}_{\semitoric}^N(\R^2)}
\newcommand{\stpolygnprime}{\mathrm{Polyg}_{\semitoric}^{N'}(\R^2)}
\newcommand{\pstpolyg}{\stpolyg_0}
\newcommand{\pstpolygn}{\stpolygn_0}
\newcommand{\pstpolygmfk}{\mathrm{Polyg}_{\semitoric}^{\mf, \vec{k}}(\R^2)_0}
\newcommand{\stpolygmfk}{\mathrm{Polyg}_{\semitoric}^{\mf, \vec{k}}(\R^2)}
\newcommand{\ingred}{\mathbb{I}}
\newcommand{\ingredmfk}{\mathbb{I}_{\mf, \vec{k}}}

\begin{document}

\title{{\bf Symplectic $G$\--capacities and integrable systems}}

\author{Alessio Figalli \,\,\,\,\,\,\, Joseph Palmer \,\,\,\,\,\,\, \'Alvaro Pelayo}
\date{}
\maketitle

\begin{abstract}
For any Lie group $G$, we construct a $G$\--equivariant analogue of
symplectic capacities and give examples when $G=\T^k\times\R^{d-k}$, in
which case the capacity is an invariant of integrable systems.
Then we study the continuity of these capacities, using the natural topologies
on the symplectic $G$\--categories on which they are defined.
\end{abstract}

\section{Introduction}\label{sec_intro}
In the 1980s Gromov proved the 
symplectic non-squeezing theorem~\cite{Gr1985}.
This influential result says 
that a ball of radius $r>0$ can be symplectically embedded 
into a cylinder of radius $R>0$ only if $r\leq R$.  This 
led to the first symplectic 
capacity, the \emph{Gromov radius}, which is the radius of the largest  ball of the same dimension which 
can be symplectically embedded into a symplectic manifold $(M,\omega)$. Symplectic capacities are a class of 
symplectic invariants introduced by Ekeland and Hofer~\cite{EkHo1989, Hocap1990}. 

In this paper we give a notion of symplectic
capacity for symplectic $G$\--manifolds, where $G$ is any Lie group,
which we call a \emph{symplectic $G$\--capacity}, and give
nontrivial examples.
Such a capacity retains
the properties of a symplectic capacity
(monotonicity, conformality, and an analogue of non-triviality)
with respect to symplectic $G$\--embeddings.
Symplectic capacities are examples of symplectic $G$\--capacities
in the case that $G$ is trivial. 
In analogy with symplectic capacities, symplectic $G$\--capacities 
distinguish the symplectic $G$\--type of symplectic 
$G$\--manifolds.
As a first example we construct an equivariant analogue of the Gromov radius
where $G=\R^k$ as follows.
Let $\sympG$ denote the category of $2n$\--dimensional symplectic $G$\--manifolds.
That is, an element of $\sympG$ is a triple $(M,\om,\phi)$ where $(M,\om)$
is a symplectic manifold and $\phi\colon G\times M \to M$ is a symplectic
$G$\--action.
Given integers $0\leq k\leq m\leq n$ we define
the \emph{$(m,k)$\--equivariant Gromov radius}
\begin{align}
 \label{eqn_eqvgromov} \cBmk \colon &\sympplain^{2n, \R^k}  \to [0,\infty]\\
 \nonumber &(M,\om,\phi) \mapsto\mathrm{sup} \{\,r>0 \mid \mathrm{B}^{2m}(r) \immrk M \,\},
\end{align}
where $\immrk$ denotes a symplectic $\R^k$\--embedding
and $\mathrm{B}^{2m}(r)\subset\C^m$ is the standard $2m$\--dimensional ball
of radius $r>0$ with $\R^k$\--action
given by rotation of the first $k$ coordinates.

\begin{prop}\label{propintro_equivgromov}
 If $k\geq 1$, the $(m,k)$\--equivariant Gromov radius 
 $\cBmk \colon \sympplain^{2n, \R^k} \to [0,\infty]$
 is a symplectic $\R^k$\--capacity.
\end{prop}
We prove Proposition~\ref{propintro_equivgromov} in Section~\ref{sec_gromov}.
Thanks to the added structure of the $\R^k$\--action the proof is elementary.
As an application of symplectic $G$\--capacities to integrable systems we define
the \emph{toric packing capacity}
\begin{align}
 \toricpack\colon  \sympT &\to [0, \infty]\label{eqn_toricpack}\\
                   (M, \om, \phi) &\mapsto \left( \frac{\sup\{\,\vol(P) \mid P \textrm{ is a toric ball packing of }M\,\}}{\vol(\mathrm{B}^{2n})}\right)^{\frac{1}{2n}},\nonumber
\end{align}
where $\vol(E)$ denotes the symplectic volume of a subset $E$ of a symplectic
manifold, $\mathrm{B}^{2n}$ is the standard symplectic unit $2n$\--ball, $\sympT$ is 
the category of $2n$\--dimensional symplectic toric manifolds,
and a toric ball packing $P$ of $M$ is given by a disjoint collection
of symplecticly and $\T^n$\--equivariantly embedded balls.
In analogy we define the \emph{semitoric packing capacity}
\[
 \semitoricpack\colon  \sympST \to [0,\infty]
\]
on $\sympST$, the category of semitoric manifolds~\cite{PeVNconstruct2011},
where $P$ in~\eqref{eqn_toricpack} is replaced by a semitoric ball packing of $M$
(Definition \ref{def_stemb}).

\begin{prop}\label{propintro_packingcaps}
 The toric packing capacity $\toricpack\colon\sympT\to[0,\infty]$ is a
 symplectic $\T^n$\--capacity and
 the semitoric packing capacity $\semitoricpack\colon\sympST\to[0,\infty]$
 is a symplectic $(S^1\times\R)$\--capacity.
\end{prop}

The  continuity of symplectic capacities is
discussed  in~\cite{Ba1995, CiHoLaSc2007, EkHo1989, ZeZi2013}.
The semitoric and toric packing capacities are each defined
on categories of integrable systems which have a natural
topology~\cite{PaSTMetric2015, PePRS2013}, but we can only discuss the continuity of
the $(m,k)$\--equivariant Gromov radius on a subcategory of its domain which has a topology, so
we restrict to the case of $(m,k)=(n,n)$.
The $\T^n$\--action on a symplectic toric manifold may be lifted to
an action of $\R^n$.
Let $\sympplain^{2n,\R^n}_\mathrm{T}$ be the symplectic category
of symplectic toric manifolds each of which is endowed with the
$\R^n$\--action obtained by lifting the given $\T^n$\--action which is a subcategory of 
$\sympplain^{2n,\R^n}$.

\begin{theorem}[Continuity of capacities]\label{thmintro}
 The following hold:
 \begin{enumerate}[label=\textup{(\roman*)}]
  \item \label{thmintro_part_toric}
    The toric packing capacity 
    $\toricpack\colon\sympT\to [0,\infty]$ is everywhere
    discontinuous and the restriction of $\toricpack$
    to the space of symplectic toric $2n$\--dimensional manifolds
    with exactly $N$ fixed points of the
    $\T^n$\--action is continuous for any choice of $N\geq 0$;
  \item \label{thmintro_part_semitoric}
    The semitoric packing capacity $\semitoricpack\colon\sympST\to [0,\infty]$
    is everywhere discontinuous
    and the restriction of $\semitoricpack$ to the space of
    semitoric manifolds with exactly $N$ elliptic-elliptic fixed points of the
    associated $(S^1\times\R)$\--action is continuous for any choice of $N\geq 0$;
  \item \label{thmintro_part_gromov}
    The $(n,n)$\--equivariant Gromov radius restricted to the space of symplectic toric manifolds
    \[\cBnn|_{\sympplain^{2n,\R^n}_\mathrm{T}}\colon\sympplain^{2n,\R^n}_\mathrm{T}\to[0,\infty]\]
    is everywhere discontinuous
    and the restriction of
    $\cBnn|_{\sympplain^{2n,\R^n}_\mathrm{T}}$ to the space of
    symplectic toric $2n$\--dimensional manifolds
    with exactly $N$ fixed points of the
    $\R^n$\--action is continuous for any choice of $N\geq 0$.
 \end{enumerate}

\end{theorem}

Theorem~\ref{thmintro} generalizes~\cite[Theorem A]{FiPe2014},
which deals with $4$\--manifolds, and solves~\cite[Problem 30]{PePRS2013}.

 In Section~\ref{sec_symplGcap} we give a general notion of symplectic $G$\--capacities 
 and we prove that the $(m,k)$\--equivariant Gromov radius
 is a capacity. In Section~\ref{sec_HamilTRactions} we review facts about Hamiltonian actions 
 and their relation to integrable systems that will be needed in the remainder of the paper.
 Sections~\ref{sec_Tncap} and~\ref{sec_symplSRcap} are devoted to constructing
 nontrivial symplectic $G$\--capacities when $G=\T^k\times\R^{d-k}$, which include the
 toric and semitoric packing capacities. In Sections~\ref{sec_toricpackcont} 
 and~\ref{sec_stmetric} we discuss the continuity of these symplectic $G$\--capacities.

\medskip
   
  \emph{Acknowledgements}. 
  The first author is supported by NSF grants DMS-1262411 and DMS-1361122.
 The  second and third authors are  supported by NSF grants DMS-1055897 and DMS-1518420.

\section{Symplectic \tex{$G$}\--capacities}
\label{sec_symplGcap}

For  $n\geq 1$ and $r>0$ let 
$\mathrm{B}^{2n}(r) \subset \C^n$ be the $2n$\--dimensional open symplectic ball 
of radius $r$ and let 
\[
 \mathrm{Z}^{2n}(r) = \{\,(z_i)_{i=1}^n\in\C^n \mid \abs{z_1} < r\,\}
\]
be the $2n$\--dimensional open symplectic cylinder of radius $r$.
Both inherit a
symplectic structure from their embedding as a subset 
of $\C^n$ with symplectic form  $\om_0 = \frac{\mathrm{i}}{2}\sum_{j=1}^{n}{\rm d}z_j\wedge {\rm d}\bar{z}_j$.
We write $\mathrm{B}^{2n} = \mathrm{B}^{2n}(1)$, 
$\mathrm{Z}^{2n} = \mathrm{Z}^{2n}(1)$, and use $\imm$ to denote a symplectic embedding.

\subsection{Symplectic capacities}

Let $\symp$ be the category of symplectic 
$2n$\--dimensional manifolds with symplectic embeddings as morphisms.
A \emph{symplectic category} is a subcategory
$\mathcal{C}$ of $\symp$ such that $(M,\om)\in\mathcal{C}$ implies 
$(M,\la \om)\in\mathcal{C}$ for all $\la\in\R\setminus\{0\}$.  Let $\mathcal{C}\subset\symp$ be a symplectic category.

The following fundamental notion of symplectic invariant
is due to Ekeland and Hofer.
\begin{definition}[\cite{EkHo1989,Hocap1990}]\label{def_cap}

 A \emph{generalized symplectic capacity} on $\mathcal{C}$ is a map 
 $c\colon \mathcal{C}\to[0,\infty]$ satisfying:
 \begin{enumerate}
  \item \label{part_monocap} \emph{Monotonicity}:  if $(M, \om), (M',\om')\in\mathcal{C}$ and there
   exists a symplectic embedding $M\imm M'$ then $c(M,\om)\leq c(M',\om')$;
  \item \label{part_confcap} \emph{Conformality}: if $\la\in\R\setminus\{0\}$ and $(M,\om)\in\mathcal{C}$
   then $c(M,\la\om) = \abs{\la} c(M,\om)$.
 \end{enumerate}
  If additionally $\mathrm{B}^{2n}, \mathrm{Z}^{2n}\in\mathcal{C}$ and $c$ satisfies:
 \begin{enumerate}[resume]
  \item \emph{Non-triviality}: $0<c(\mathrm{Z}^{2n}, \om_0)<\infty$ 
   and  $0<c(\mathrm{B}^{2n}, \om_0)<\infty$;
 \end{enumerate}
 then $c$ is a \emph{symplectic capacity}. 
\end{definition}

\subsection{Symplectic \tex{$G$}\--capacities} 
\label{sec_subsymplGcap}

Let $G$ be a Lie group and let
$\mathrm{Sympl}(M)$ denote the group of symplectomorphisms of the symplectic manifold $(M,\omega)$.
A smooth $G$\--action  $\phi\colon  G\times M \to M$ is \emph{symplectic} if 
$\phi(g,\cdot)\in\mathrm{Sympl}(M)$ for each $g\in G$.
The triple $(M, \om, \phi)$  is a
\emph{symplectic $G$\--manifold}. 
A \emph{symplectic $G$\--embedding}
$\rho \colon (M_1,\omega_1,\phi_1) \imm (M_2,\omega_2,\phi_2)$
 is a symplectic embedding for which there exists an 
  automorphism $\La\colon G\to G$ of $G$ such that
  $
   \rho(\phi_1(g, p)) = \phi_2(\La(g), \rho(p))
  $
  for all $p\in M_1, \, g\in G$,
  in which case we say that $\rho$ is a symplectic $G$\--embedding
  with respect to $\La$.
  We write $\immG$ to denote a 
symplectic $G$\--embedding. We denote the collection of all 
$2n$\--dimensional symplectic $G$\--manifolds by $\sympG$.
The set $\sympG$ is a category with morphisms given
by symplectic $G$\--embeddings. We call a subcategory $\mathcal{C}_G$
of $\sympG$ a \emph{symplectic $G$\--category} if $(M,\om, \phi)\in\mathcal{C}_G$ implies
$(M,\la \om, \phi)\in\mathcal{C}_G$ for any $\la\in\R\setminus\{0\}$.
Let $\mathcal{C}_G\subset\sympG$ be a symplectic $G$\--category.

\begin{definition}\label{def_eqvcap}
 A \emph{generalized symplectic
 $G$\--capacity} on $\mathcal{C}_G$ is a map $c\colon \mathcal{C}_G\to[0,\infty]$ satisfying:
 \begin{enumerate}
  \item \label{part_mono} \emph{Monotonicity}:  if 
   $(M,\om,\phi),(M',\om',\phi')\in\mathcal{C}_G$ and there exists a 
   symplectic $G$\--embedding $M\immG M'$ then $c(M,\om,\phi)\leq c(M',\om',\phi')$;
  \item \label{part_conf} \emph{Conformality}: if $\la\in\R\setminus\{0\}$ and 
   $(M,\om,\phi)\in\mathcal{C}_G$ then $c(M,\la\om, \phi) = \abs{\la} c(M,\om,\phi)$.
 \end{enumerate}
 \end{definition} 
 
 When the symplectic form and $G$\--action are understood we often write $c(M)$ for $c(M,\om,\phi)$.
 Let $c$ be a generalized symplectic $G$\--capacity on a symplectic $G$\--category $\mathcal{C}_G$.
 
 \begin{definition}
  For $(N,\om_N,\phi_N)\in\mathcal{C}_G$ we say that $c$ satisfies \emph{$N$\--non-triviality}
  or is \emph{non-trivial on $N$} if $0<c(N)<\infty$.
 \end{definition}
  
 \begin{definition}\label{def_tame}
 We say that $c$ is \emph{tamed} by $(N,\om_N,\phi_N)\in\sympG$ if
 there exists
 some $a\in(0,\infty)$ such that the following two properties hold:\\
 (1) if $M\in\mathcal{C}_G$ and there exists a symplectic
   $G$\--embedding $M\immG N$ then $c(M)\leq a$;\\
 (2) if $P\in\mathcal{C}_G$ and there exists a symplectic
   $G$\--embedding $N\immG P$ then $a\leq c(P)$.
 \end{definition}

 The non-triviality condition in Definition~\ref{def_cap}
 requires that $\mathrm{B}^{2n},\mathrm{Z}^{2n}\in\mathcal{C}_G$ and $0<c(\mathrm{B}^{2n})\leq c(\mathrm{Z}^{2n}) < \infty$, and 
 tameness encodes this second condition without necessarily including the first one. 
   If $c$ is a generalized symplectic $G$\--capacity on $\mathcal{C}_G\subset\sympG$ we define
  \begin{align*}
   \sympG_0(c) &= \{\,N\in\sympG \mid \inf\{\,c(P) \mid P\in\mathcal{C}_G, N\immG P\,\}=0\,\},\\
   \sympG_\infty(c) &= \{\,N\in\sympG \mid \sup\{\,c(M) \mid M\in\mathcal{C}_G, M\immG N\,\}=\infty\,\},\\
   \sympG_{\mathrm{tame}}(c) &= \{\,N\in\sympG \mid \textrm{$c$ is tamed by $N$}\,\}.
  \end{align*} 
 A generalized symplectic $G$\--capacity gives rise to a decomposition of $\sympG$.
  
 \begin{prop}
  Let $c$ be a generalized symplectic $G$\--capacity on a symplectic $G$\--category $\mathcal{C}_G$. Then:
  \begin{enumerate}[label = (\alph*), font = \normalfont]
   \item\label{item_ma} $\sympG = \sympG_0(c) \cup \sympG_\infty(c) \cup \sympG_{\mathrm{tame}}(c)$;
   \item\label{item_mb} the union in part~\textup{\ref{item_ma}} is pairwise disjoint;
   \item\label{item_mlemma} $c$ is non-trivial on $N\in\sympG$ if and only if $N\in\mathcal{C}_G\cap\sympG_{\mathrm{tame}}(c)$.
  \end{enumerate}
 \end{prop}
 
 \begin{proof}
  In order to prove item~\ref{item_ma} we show that if $N\in\sympG$ is not in 
  $\sympG_0(c) \cup \sympG_\infty(c)$ then it is in 
  $\sympG_{\mathrm{tame}}(c)$.  If
  $M\immG N \immG P$ for some $M,P\in\mathcal{C}_G$ then
  $M\immG P$ so $c(M)\leq c(P)$.  Let
  $
   a_1 = \sup\{\,c(M) \mid M\immG N\,\}
  $
  and
  $
   a_2 = \inf\{\,c(P) \mid N\immG P\,\}.
  $
  Since $N\notin \sympG_0(c) \cup \sympG_\infty(c)$
  we have that $0<a_1\leq a_2 < \infty$.  Pick $a\in[a_1,a_2]$.
  If $M\in\mathcal{C}_G$ and $M\immG N$ then $c(M)\leq a_1 \leq a$
  and if $P\in\mathcal{C}_G$ and $N\immG P$ then $c(P)\geq a_2 \geq a$
  so $N\in\sympG_{\mathrm{tame}}(c)$.  Item~\ref{item_mb} follows from a similar argument
  and~\ref{item_mlemma} is immediate.
 \end{proof}

 In light of item~\ref{item_mlemma} we view $\sympG_{\mathrm{tame}}(c)$ as an
 extension of the set of elements of $\sympG$ on which $c$ is non-trivial
 to include those elements outside of the domain of $c$.

\subsection{Symplectic \tex{$(\T^k\times\R^{d-k})$}\--capacities}
\label{sec_symplTRactions}

For $1\leq d\leq n $ the standard action of $\T^d$ on $\C^n$ is given by 
 \[
 \phi_{\C^n}\big((\al_i)_{i=1}^d, (z_i)_{i=1}^n\big) = (\al_1 z_1, \ldots, \al_d z_d, z_{d+1}, \ldots, z_n).
 \]
This action induces actions of $\T^d=\T^k\times\T^{d-k}$
on $\mathrm{B}^{2n}$ and $\mathrm{Z}^{2n}$, which in turn induce
the standard actions of
$\T^k\times\R^{d-k}$ on $\mathrm{B}^{2n}$ and $\mathrm{Z}^{2n}$
for  $k\leq d$. The action of an
element of $\T^k\times\R^{d-k}$ is the
action of its image under the quotient map
$\T^k\times\R^{d-k} \to \T^d$.
In the following we endow $\mathrm{B}^{2n}$ and $\mathrm{Z}^{2n}$
with the standard actions.

 \begin{definition}\label{def_RTcapacity}
 A generalized symplectic $(\T^k\times\R^{d-k})$\--capacity is a \emph{symplectic 
 $(\T^k\times\R^{d-k})$\--capacity} if it is
 tamed by $\mathrm{B}^{2n}$ and $\mathrm{Z}^{2n}$.
 \end{definition}

\subsection{A first example}\label{sec_gromov}

 The \emph{Gromov radius} $\cgromov \colon  \symp \to (0,\infty]$ is given by
 $$\cgromov(M):=\sup\{\,r>0 \mid \mathrm{B}^{2n}(r) \imm M\,\}.$$ 
 Fix  $0\leq k \leq m \leq n$ and let $\cBmk$ be as in Equation \eqref{eqn_eqvgromov}.  
 If $k=0$ and $m=n$ then $\cgromov = \cBmk$. 

\begin{proof}[Proof of Proposition~\ref{propintro_equivgromov}]
  Parts \eqref{part_mono} and \eqref{part_conf} of Definition~\ref{def_eqvcap}
  are immediate. By the standard inclusion map $\cBmk (\mathrm{B}^{2n})\geq 1$ so we only must show 
  that $\cBmk(\mathrm{Z}^{2n})\leq 1$.
  Suppose that for $r>1$ $\rho\colon  \mathrm{B}^{2m}(r) \immrk \mathrm{Z}^{2n}$ is a 
  symplectic $\R^k$\--embedding with respect to some $\La\in\mathrm{Aut}(\R^k)$.
  Let 
  \[(\eta_1, \ldots, \eta_k)=\La^{-1}(1, 0, \ldots, 0).\]
  Since $\La$ is an automorphism $\eta_{j_0} \neq 0$ for some
  $j_0\in\{1, \ldots, k\}$.
  Pick \[w=(0, \ldots, 0, w_{j_0}, 0, \ldots, 0)\in \mathrm{B}^{2m}(r)\] with entries all zero
  except in the $j_{0}^{\mathrm{th}}$ position and such that $\abs{w_{j_0}} > 1$.
  Let $u = (u_1, \ldots, u_n) = \rho(w)$ and note $\abs{u_1}<1$.
  Let $\io\colon\R\imm\R^k$ be given by $\io(x) = (x, 0, \ldots, 0)$.
  Let $\phi_\mathrm{B}\colon \R^k\times \mathrm{B}^{2m}(r)\to \mathrm{B}^{2m}(r)$ and 
  $\phi_\mathrm{Z}\colon \R^k\times \mathrm{Z}^{2n}\to \mathrm{Z}^{2n}$ be the standard actions of $\R^k$.
  Then for $x\in\R$
  \[
   \rho\big(\phi_\mathrm{B}(\La^{-1}\circ\io (x), w)\big) = \phi_\mathrm{Z}(\io(x),\rho(w)) = \phi_\mathrm{Z}(\io(x), u).
  \]
  Thus 
  \begin{equation}\label{eqn_gromovproof}
   \rho\big(\{\, (0,\ldots, \mathrm{e}^{2\mathrm{i}x\eta_{j_0}}w_{j_0}, 0, \ldots,0) \mid x\in \R \,\}\big) = \{\, (\mathrm{e}^{2\mathrm{i}x}u_1, u_2, \ldots, u_n)\mid x\in\R \,\}
  \end{equation}
  and since $\rho$ is injective and $\eta_{j_0}\neq0$ this means that $u_1\neq0$.
  Let \[ S_\mathrm{B} = \{\, (0, \ldots, 0, \al, 0, \ldots, 0)\in \mathrm{B}^{2m}(r) \mid \, \abs{\al} < \abs{w_{j_0}}\,\}\] 
  where $\al$ is in the $j_0^{\mathrm{th}}$ position and
  \[ S_\mathrm{Z} = \{\, (\be, u_2, \ldots, u_n)\in \mathrm{Z}^{2n} \mid\, \abs{\be} < \abs{u_1}\,\}.\] 
  Equation~\eqref{eqn_gromovproof} implies that $\rho(\partial S_\mathrm{B} ) = \partial S_\mathrm{Z}$
  and since $\rho$ is an embedding this means $\partial (\rho(S_\mathrm{B}))=\partial S_\mathrm{Z}$.
  Since $\rho (S_\mathrm{B})$ and
  $S_\mathrm{Z}$ have the same boundary, $\om_\mathrm{Z}$ is closed, and $\mathrm{Z}^{2n}$ has 
  trivial second homotopy group, 
  $$\varint_{\rho (S_\mathrm{B})} \om_\mathrm{Z} = \varint_{S_\mathrm{Z}} \om_\mathrm{Z}.$$
  Finally, integrating over $z$ we have
  \[
   \frac{\mathrm{i}}{2}\varint_{\abs{z}<\abs{w_j}} \mathrm{d}z\wedge \mathrm{d}\bar{z} = \varint_{S_\mathrm{B}} w_\mathrm{B} 
                                                =  \varint_{S_\mathrm{B}} \rho^* \om_\mathrm{Z}
                                                = \varint_{\rho(S_\mathrm{B})} \om_\mathrm{Z}
                                                =  \varint_{S_\mathrm{Z}} \om_\mathrm{Z}
                                                = \frac{\mathrm{i}}{2} \varint_{\abs{z}<\abs{u_1}} \mathrm{d}z\wedge \mathrm{d}\bar{z}.
  \]
  This implies that $1 < \abs{w_j} = \abs{u_1} < 1$, which is a contradiction.
 
 \end{proof}
 
It follows from the proof that 
$\cBmk(\mathrm{B}^{2n}) = \cBmk(\mathrm{Z}^{2n}) = 1$.

\begin{prop}
Let $M=(S^2)^n$ with symplectic form
$\om_M = \frac{1}{2}\sum_{i=1}^n\mathrm{d}h_i\wedge\mathrm{d}\te_i$
where $h_i\in[-1,1]$, $\te_i\in[0,2\pi)$, $i=1,\ldots,n$, are the
standard height and angle coordinates.
Let $\R^k$, $1\leq k \leq n$,  act on $M$ by rotating the first $k$ components.
Then 
\[\cBmk(M)=\sqrt{2}\]
for all $m,k\in\Z$ with
$1\leq k \leq m \leq n$.
\end{prop}

\begin{proof}
The map $\rho\colon \mathrm{B}^{2n}(\sqrt{2})\immrn M$ given by
\[
 \rho(r_1 e^{\mathrm{i}\te_1},\ldots, r_n e^{\mathrm{i}\te_n})= (\te_1, r_1^2-1, \ldots , \te_n, r_n^2-1 )
\]
is a symplectic $\R^n$\--embedding, so
$\cB{n}{n}(M)\geq \sqrt{2}$. 

Fix $k,m,n\in\Z$ satisfying $0< k \leq m \leq n$ and let 
$\rho:\mathrm{B}^{2m}(r) \immrk M$ be a symplectic $\R^k$\--embedding
for some $r>0$.  
Let 
\[
 B_j = \{\, (h_i, \te_i)_{i=1}^n\in M \mid h_i \in\{\pm1\} \textrm{ if $i\leq k$ and $i\neq j$}\,\}
\]
for $j=1, \ldots, k$.  For $R\in(0,r)$ let 
\[
 A_R = \{\, (z, 0, \ldots, 0)\in \mathrm{B}^{2m}(r) \mid\, \abs{z}<R \,\}.
\]
Every point in $A_R$, except at the identity, has the same $(k-1)$\--dimensional 
stabilizer in $\R^k$ so there exists $j_0\leq k$ such that
$\rho (A_R)\subset B_{j_0}$ for all $R\in(0,r)$. 
Write $\rho = (H_i,\Theta_i)_{i=1}^n$ and consider coordinates $(r,\te)$ on $A_R$ given
by $(r\mathrm{e}^{\mathrm{i}\te}, 0, \ldots, 0)\to(r,\te)$.
For $i\neq j_0$ this means that $H_i$ is constant if $i\leq k$ and 
the $\R^n$\--equivariance of $\rho$ implies that $H_i$ and $\Theta_i$ are
independent of $\te$ if $i>k$.
Thus if $i\in\{1, \ldots, n\}$ and $i\neq j_0$ then
\[
 \varint_{\rho(A_R)} \mathrm{d}h_i\wedge \mathrm{d}\te_i = \varint_{A_R} \mathrm{d}H_i\wedge\mathrm{d}\Theta_i = 0
\]
for $R\in (0,r)$.
Therefore,
\[
  \pi R^2 = \smashoperator{\varint_{A_R}} \om_\mathrm{B} = \smashoperator{\varint_{\rho(A_R)}}\om_M = \frac{1}{2}\varint_{\rho(A_R)} \mathrm{d}h_{j_0}\wedge\mathrm{d}\te_{j_0} + \frac{1}{2}\sum_{i\neq j_0} \left(\varint_{\rho(A_R)}\mathrm{d}h_i\wedge\mathrm{d}\te_i\right)\leq \frac{1}{2}\varint_{S^2}\mathrm{d}h\wedge\mathrm{d}\te = 2\pi
\]
for any $R\in(0,r)$.
This implies that $r\leq\sqrt{2}$ so 
\[
 \sqrt{2} \leq \cBnn(M) \leq \cBmk(M) \leq \sqrt{2}
\]
for any $k,m,n\in\Z$ satisfying $0<k\leq m\leq n$.
\end{proof}

\begin{figure}[ht]
 \centering
  \includegraphics[height=120pt]{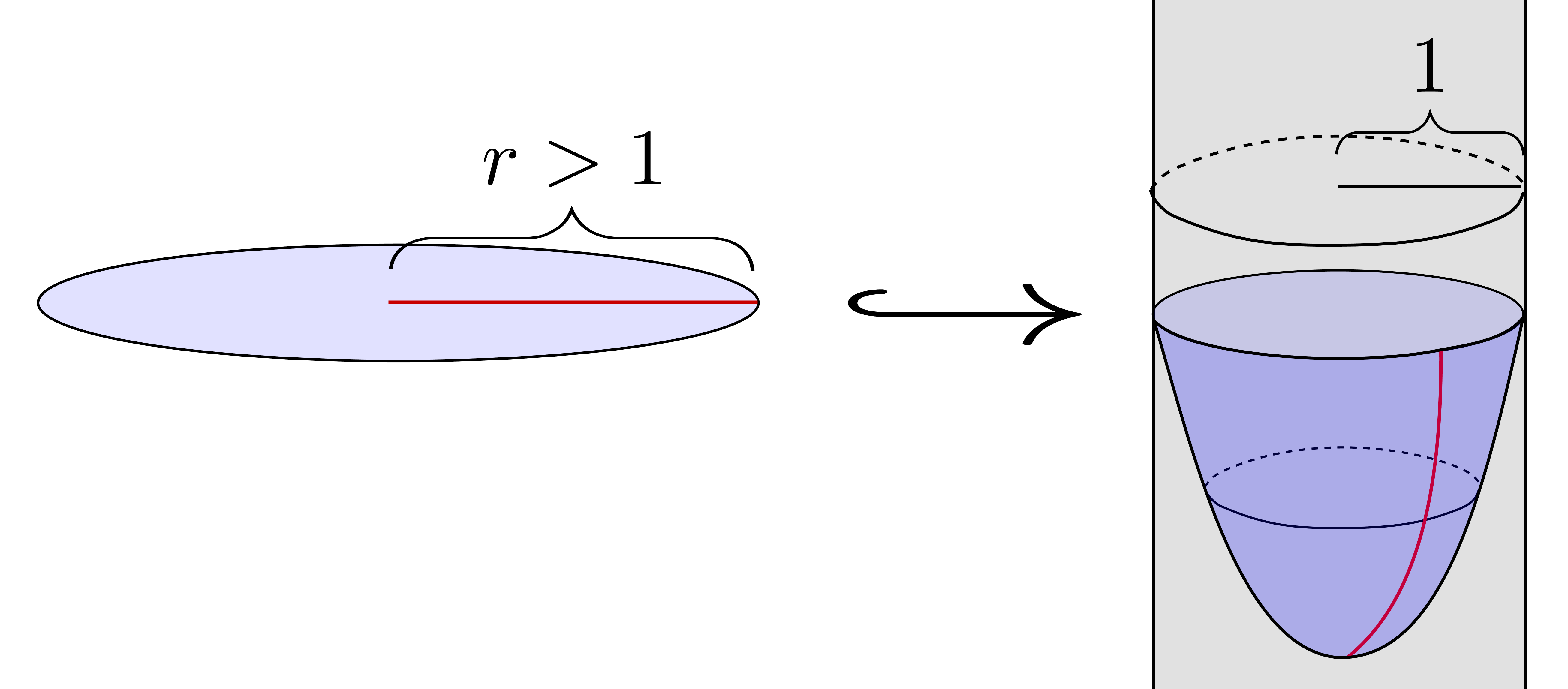}
  \caption{A symplectic $\R$\--embedding.}
  \label{fig_embdisk}
 \end{figure}

\begin{example}\label{ex_z}
For $k, n\in\Z_{>0}$ with $k<n$ let $M=\mathrm{Z}^{2n}$ with
the standard symplectic form.
There are two natural ways in which $\R^k$
can act symplectically on $M$ given by
\[
 \phi_1((t_i)_{i=1}^k, (z_i)_{i=1}^n) = (\mathrm{e}^{2\mathrm{i}t_1} z_1, \mathrm{e}^{2\mathrm{i}t_2} z_2, \ldots, \mathrm{e}^{2\mathrm{i}t_k} z_{k}, z_{k+1}, \ldots, z_{n})
\]
and 
\[
 \phi_2((t)_{i=1}^k, (z_i)_{i=1}^n) = (z_1, \mathrm{e}^{2\mathrm{i}t_1} z_2, \ldots, \mathrm{e}^{2\mathrm{i}t_k} z_{k+1}, z_{k+2}, \ldots, z_{n})
\]
where $\phi_i\colon  \R^k\times M\to M$ for $i=1,2$.
Let $\rho:M\to M$ be given by
\[
 \rho((z_i)_{i=1}^n) = \left(\frac{z_{k+1}}{1+\abs{z_{k+1}}}, \frac{z_1}{1-\abs{z_1}}, z_2, \ldots, z_{k}, z_{k+2}, \ldots, z_{n}\right)
\]
similar to the map shown in Figure \ref{fig_embdisk}. The map $\rho$ is well-defined because $\abs{z_1}<1$ and it
is an $\R^k$\--equivariant diffeomorphism because
\begin{align*}
 \rho\big(\phi_1((t_i)_{i=1}^k, (z_i)_{i=1}^n)\big)
  &=\left(\frac{z_{k+1}}{1+\abs{z_{k+1}}}, \mathrm{e}^{2\mathrm{i}t_1}\frac{z_1}{1-\abs{z_1}}, \mathrm{e}^{2\mathrm{i}t_2} z_2, \ldots, \mathrm{e}^{2\mathrm{i}t_k} z_{k}, z_{k+2}, \ldots, z_{n}\right)\\
  &=\phi_2\left((t)_{i=1}^k, \rho((z_i)_{i=1}^n)\right)
\end{align*}
for all $t_1, \ldots, t_k\in\R$.
Thus the symplectic $\R^k$\--manifolds
$(M,\om,\phi_1)$ and $(M,\om,\phi_2)$ are symplectomorphic via
the identity map and
$\R^k$\--equivariantly diffeomorphic via $\rho$ but they are
not $\R^k$\--equivariantly symplectomorphic
because $\cBoo(M, \om, \phi_1) = 1$ and $\cBoo(M, \om, \phi_2)=\infty$.
\end{example}

\section{Hamiltonian \tex{$(\T^k \times \R^{n-k})$}\--actions}\label{sec_HamilTRactions}

 In this section we review the facts we need for the remainder of the paper 
 about Hamiltonian  $(\T^k \times \R^{n-k})$\--actions and their relation to toric and semitoric
 systems.
 Let $(M,\om)$ be a symplectic manifold and $G$ a Lie group with Lie algebra $\mathrm{Lie}(G)$
 and dual Lie algebra $\mathrm{Lie}(G)^*$.
 A symplectic $G$\--action is \emph{Hamiltonian} if there exists a map $\mu\colon M\to\mathrm{Lie}(G)^*$,
 known as the \emph{momentum map}, such that 
 \[-\mathrm{d}\langle \mu, \mathcal{X} \rangle = \om (\mathcal{X}_M, \cdot)\]
 for all $\mathcal{X}\in\mathrm{Lie}(G)$ where $\mathcal{X}_M$ denotes the vector field on $M$ generated by $\mathcal{X}$ via
 the action of $G$.
 A \emph{Hamiltonian $G$\--manifold} is a quadruple $(M,\om,\phi,\mu)$ where $(M,\om,\phi)$ is a symplectic
 $G$\--manifold for which the action of $G$ is Hamiltonian with momentum map $\mu$.
 Let $\hamG$ denote the category of $2n$\--dimensional Hamiltonian $G$\--manifolds with morphisms given
 by symplectic $G$\--embeddings which intertwine the momentum maps.
 Given $f\colon M\to\R$ the associated Hamiltonian vector field
 is the vector field $\mathcal{X}_f$ on $M$ satisfying $\om(\mathcal{X}_f,\cdot) = -\mathrm{d}f$.
 \begin{definition}\label{def_intsystem}
  An \emph{integrable system} is a triple $(M,\om, F)$ where $(M, \om)$
  is a $2n$\--dimensional symplectic manifold and 
  $F=(f_1, \ldots, f_n)\colon  M \to \R^n$ is a smooth map such that 
  $f_1, \ldots, f_n$ pairwise  Poisson commute,
  i.e. $\om(\mathcal{X}_{f_i},\mathcal{X}_{f_j})=0$ for all $i,j=1, \ldots, n$,
  and the Hamiltonian vector fields
  $(\mathcal{X}_{f_1})_p, \ldots, (\mathcal{X}_{f_n})_p$ 
  are linearly independent for almost all $p\in M$.
 \end{definition}
 Let $\mathcal{I}^{2n}$
 denote the set of all $2n$\--dimensional integrable systems
 and define an equivalence relation $\sim_\mathcal{I}$ on this space
 by declaring $(M,\om, F)$ and $(M',\om',F')$ to be equivalent
 if there exists a symplectomorphism $\phi\colon M\to M'$ such that 
 $F-\phi^* F'\colon M\to\R^n$ is constant.
 
 \subsection{Hamiltonian \tex{$\R^n$}\--actions and integrable systems}
 \label{sec_RnandIS}
 
 Let $(M,\om,F=(f_1, \ldots, f_n))$ be an integrable system and
 for $i=1, \ldots, n$ let $\psi^t_i\colon M\to M$ denote the flow
 along $\mathcal{X}_{f_i}$. The \emph{Hamiltonian flow
 action} $\phi_F\colon \R^n\times M\to M$, given by
 $\phi_F((t_1, \ldots, t_n), p) = \psi_1^{t_1}\circ \ldots \circ \psi_n^{t_n}(p)$,
 defines a Hamiltonian $\R^n$\--action on $M$.
 The action of $G$ on $M$ is \emph{almost
 everywhere locally free} if the stabilizer of $p$
 is discrete for almost all $p\in M$. 
 Let ${\mathcal{F}}\sympplain^{2n, \R^n}$ be the space of
 $\R^n$\--manifolds on which the action of $\R^n$ is 
 Hamiltonian and almost everywhere locally free
 and let ${\sim}_{\R^n}$ denote equivalence
 by $\R^n$\--equivariant symplectomorphisms.

  \begin{lemma}\label{lem_grpact_technical}
  Let $\mathcal{X}_1, \ldots, \mathcal{X}_n$ be vector fields with commuting flows on
  an $m$\--manifold $M$, with $n\leq m$.
  Let $\R^n$ act on $M$ by 
  $
   \phi((t_1, \ldots, t_n), p) = \psi^{t_1}_1 \circ \ldots \circ \psi^{t_n}_n (p)
  $
  where $\psi^t_i$ is the flow of $\mathcal{X}_i$.
  Then, for $p\in M$, the vectors $(\mathcal{X}_1)_p,\ldots,(\mathcal{X}_n)_p\in T_p M$
  are linearly independent if and only if the stabilizer of $p$
  under the action $\phi$ is discrete. 
 \end{lemma}

 \begin{proof}
  If $(\mathcal{X}_1)_p, \ldots, (\mathcal{X}_n)_p$ are linearly independent then, 
  since they have commuting flows, there is a chart $(U,g)$,
  with $U\subset M$ and $g:U\to\R^m$, such that $g^{-1}\colon g(U)\to U$ satisfies
  \[
   g^{-1}(t_1, \ldots, t_n, 0, \ldots, 0) = \phi((t_1, \ldots, t_n), p)
  \]
  for any $(t_1, \ldots, t_n, 0, \ldots, 0)\in g(U)$.
  Thus $g(U)$
  is an open neighborhood of the identity in $\R^n$ and 
  there exists no non-zero point in $g(U)$
  which fixes $p$, so the stabilizer of $p$ under the action of
  $\R^n$ is discrete.
  On the other hand, if $(\mathcal{X}_1)_p,\ldots, (\mathcal{X}_n)_p$ are linearly dependent, 
  there exist $t_1, \ldots, t_n\in\R$ not all zero
  such that $\sum_{i=1}^n t_i (\mathcal{X}_i)_p = 0$. Thus 
  $(\al t_1, \ldots, \al t_n) \in \R^n$ fixes $p$
  for all $\al\in\R$ and so the stabilizer of $p$ is not discrete.
 \end{proof}

 \begin{prop}\label{prop_groupactions}
  Let $\psi$ be
  the map which takes an integrable system on $M$ to 
  $M$ equipped with its Hamiltonian flow action.
  Then 
  \[
   \psi\colon \mathcal{I}^{2n}/{{\sim}_\mathcal{I}} \to {\mathcal{F}}\sympplain^{2n, \R^n}/{{\sim}_{\R^n}}
  \]
  is a bijection.
 \end{prop}
 
 \begin{proof}
  By Lemma~\ref{lem_grpact_technical}
  we know that the Hamiltonian flow action
  must be almost everywhere locally free because the Hamiltonian
  vector fields of an integrable system are by definition independent
  almost everywhere.  Next suppose that $\R^n$ acts Hamiltonianly on $M$ in such a way 
  that the action is almost everywhere locally free.
  Since the action is Hamiltonian there exists a momentum map 
  $\mu\colon M\to\mathrm{Lie}(\R^n)^*$.  Define 
  $F=(f_1, \ldots, f_n)\colon M\to \R^n$ by $F = A\circ\mu$ where
  $A\colon \mathrm{Lie}(\R^n)^*\to\R^n$
  is the standard identification which is induced by the standard basis
  $\{e_1, \ldots, e_n\}$ of $\R^n$.
  These functions Poisson commute because 
  action by the components of $\R^n$ commute and are linearly
  independent at almost all points because the group action is 
  almost everywhere locally free (Lemma~\ref{lem_grpact_technical}).
  Thus, $(M,\om, F)$ is an integrable
  Hamiltonian system as in Definition~\ref{def_intsystem}.
  Let $\{v_1, \ldots, v_n\}$ be the standard basis of
  $\mathrm{Lie}(\R^n)\cong\R^n$ induced by the standard
  basis of $\R^n$.
  Let $v_M$ denote the vector field on $M$ generated
  by $v\in\mathrm{Lie}(\R^n)$ via the action of $G$.
  Then $\langle \mu, v_i\rangle = f_i\colon M\to\R$  
  so $\mathrm{d}f_i = \om((v_i)_M, \cdot)$
  which means that the Hamiltonian vector field associated to
  $f_i$ is $(v_i)_M$. Thus the Hamiltonian flow action
  related to $F$ is the original action of $\R^n$.
 \end{proof}
  
Here we fix the identification between $\mathrm{Lie}(\T^n)^*$ and $\R^n$
that we will use for the remainder of the paper.
We specify our convention
by choosing an epimorphism from $\R$ to $\T^1$,
which we take to be $x\mapsto e^{2\sqrt{-1}x}$.

 \subsection{Hamiltonian \tex{$\T^k$}\--actions}
 Atiyah~\cite{At1982} and Guillemin-Sternberg~\cite{GuSt1982} proved that
 if $(M,\om,\phi, \mu)$ is a compact connected Hamiltonian $\T^k$\--manifold,
 then $\mu(M)\subset\mathrm{Lie}(\T^k)^*$ is the convex hull of the image of the
 fixed points of the $\T^k$\--action. The case in which $k=n$ and the torus action
 is effective enjoys very special
 properties, and in such a case $(M,\om,\phi, \mu)$  is called a
 \emph{symplectic toric manifold}, or a \emph{toric integrable system}. 
 An \emph{isomorphism} of such manifolds is a symplectomorphism which intertwines
 their respective momentum maps. We denote by $\hamT$ the category of $2n$\--dimensional 
 symplectic toric manifolds with morphisms as symplectic $\T^n$\--embeddings
 and we denote equivalence by toric isomorphism by ${\approx}_{\mathrm{T}}$.
 In general being an invariant is weaker than being monotonic, but in the case of toric
 manifolds these are equivalent because symplectic $\T^n$\--embeddings between toric manifolds are automatically 
 $\T^n$\--equivariant symplectomorphisms.
 Delzant proved~\cite{De1988} that in this case 
 $\mu(M)$ is  a \emph{Delzant polytope}, i.e. simple, rational, and smooth, and
 that
 \begin{align*}
  \Psi\colon &\mhamT \to \P_\toric\\
 &[(M, \om, \phi, \mu)] \mapsto \mu(M)
 \end{align*}
 is a bijection, where $\P_\toric$ denotes the set of $n$\--dimensional Delzant polytopes. 
 Let $\mathrm{Ham}^{2n, \T^n}\to\sympplain^{2n, \T^n}$ be given by $(M,\om,\phi,\mu)\mapsto (M,\om,\phi)$
 and let $\sympT$ denote the image of $\hamT$ under this map.
 Also let $\sim_\mathrm{T}$ denote equivalence on $\sympT$
 by $\T^n$-equivariant symplectomorphisms.

  \subsection{Hamiltonian \tex{$(S^1\times \R)$}\--actions}
  We say that an integrable system $(M, \om, F=(J,H)\colon M\to\R^2)$ is a 
  \emph{semitoric integrable system}   or \emph{semitoric manifold} if $(M,\om)$ is a $4$\--dimensional connected symplectic manifold,
  $J$ is a proper momentum map for an effective Hamiltonian $S^1$\--action on $M$, and
  $F$ has only non-degenerate singularities which have no 
  real-hyperbolic blocks (see~\cite[Section 4.2.1]{PeVNsymplthy2011}).
  A semitoric integrable system is \emph{simple} if there is at most
  one singular point of focus-focus type in $J^{-1}(x)$ for each $x\in\R$. 
  Let $(M_i, \om_i, F_i = (J_i, H_i))$ be a semitoric manifold for $i=1,2$.
  A \emph{semitoric isomorphism} between them is a symplectomorphism
  $\rho\colon M_1\to M_2$ such that $\rho^* (J_2, H_2) = (J_1, f(J_1, H_1))$
  where $f\colon\R^2\to\R$ is a smooth function for which $\deriv{f}{H_1}$
  is everywhere nonzero.
  Let $\hamST$ denote the category of simple semitoric systems
  and let ${\approx}_{\mathrm{ST}}$ denote equivalence by semitoric isomorphism.
  Let $\sympST$ denote the image of $\hamST$ under the map 
  $\mathrm{Ham}^{4,  S^1\times\R}\to\sympplain^{4, S^1\times\R}$
  given by $(M, \om, \phi, \mu)\mapsto(M,\om,\phi)$
  and let $\sim_{\mathrm{ST}}$ denote the
  equivalence on $\sympST$ inherited from $\sim_{\mathrm{ST}}$
  on $\hamST$.
  
  The number of focus-focus singular points of an integrable system must
  be finite~\cite{VN2007}, and we denote it by $m_f$.
  
 \subsubsection{Invariant of focus-focus singularities}
 It is proven in~\cite{VN2003} that the structure in the neighborhood of
 a fiber over a focus-focus point  is
 determined by a Taylor series.
  Let $\Rxy$ denote the space of real formal Taylor series in two variables $X$
  and $Y$ and let $\Rxyz\subset\Rxy$ denote the subspace of series 
  $\sum_{i,j>0}\si_{i,j}X^i Y^j$
  which have $\si_{0,0}=0$ and $\si_{0,1}\in[0, 2\pi)$. The Taylor series invariant consists of $\mf$ elements of $\Rxyz$, one
 for each focus-focus singular point.

 \subsubsection{Affine and twisting-index invariants}
  Denote the set of rational polygons in $\R^2$
  by $\polyg$. For $\la\in\R$ let
  $\ell_\la$ denote the set of $(x,y)\in\R^2$ such that $x = \la$, Let $\vertr$
  denote the collection of all $\ell_\la$ as $\la$ varies in $\R$.
  Let $\pi_i\colon\R^2\to\R$ denote the projection onto the $i^{\mathrm{th}}$
  coordinate for $i=1,2$.
  Notice that elements of $\polyg$ can be non-compact.
  A \emph{labeled weighted polygon of complexity $\mf\in\Z_{\geq0}$} is 
  an element
  $$\De_w = \big(\De, (\ell_{\la_j}, \ep_j, k_j)_{j=1}^\mf\big)\in\polyg\times\big(\vertr\times\{-1,+1\}\times\Z\big)^\mf$$
  with
  $\min_{s\in\De}\pi_1(s)<\la_1<\ldots<\la_\mf<\max_{s\in\De}\pi_1(s)$.  We denote the space of labeled weighed polygons by $\lwpolyg$. Let 
 \begin{equation}\label{eqn_T}
 T = \left(\begin{array}{cc} 1&0 \\1&1 \end{array}\right)\in\mathrm{SL}_2(\Z)
 \end{equation}
 and for $v_1, \ldots, v_n \in\Z^n$ let 
 $\det(v_1, \ldots, v_n)$ denote the determinant of the matrix with columns 
 given by $v_1, \ldots, v_n$.

 \begin{figure}[ht]
  \centering
  \includegraphics[height=160pt]{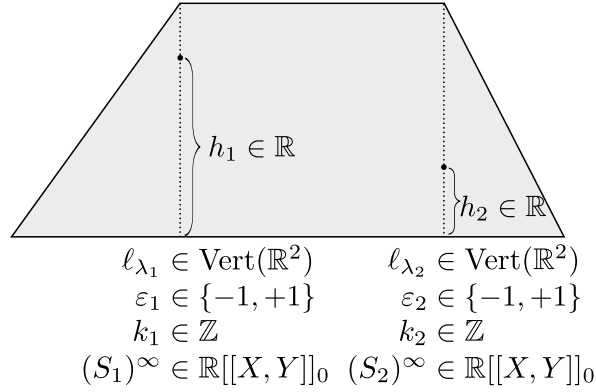}
  \caption{The complete invariant
   of a semitoric system is a collection of these objects.}
   \label{fig_primitive}
 \end{figure}
 
  The \emph{top boundary of} $\De \in \polyg$ is
  the set $\partial^{\mathrm{top}}\De$ of $(x_0, y_0)\in\De$ such that 
  $y_0$ is the maximal $y\in\R$ such that $(x_0,y)\in\De$.
  A point $p\in\partial\De$ is a \emph{vertex} of $\De$ if the 
  edges meeting at $p$ are not co-linear.  Let $p$ be 
  a vertex of $\De$ and let $u,v\in\Z^2$ be primitive vectors directing
  the edges adjacent to $p$ ordered so that $\det(u,v)>0$.
  Then we say that:\\
  -  $p$ \emph{satisfies the Delzant condition} if $\det(u,v) = 1$; \\
-  $p$ \emph{satisfies the hidden condition} if $\det(u, Tv) = 1$;\\
- $p$ \emph{satisfies the fake condition} if $\det(u, Tv) = 0$.

  We say that $\De$ has \emph{everywhere finite height} if
  $\De\cap\ell_\la$ is either compact or empty for all $\la\in\R$.
   A \emph{primitive semitoric polygon of complexity $\mf\in\Z_{\geq0}$}~\cite{KaPaPe2015}
   is a labeled weighted polygon 
   $\big(\De, (\ell_{\la_j}, \ep_j, k_j)_{j=1}^\mf\big)\in\lwpolyg$
   such that:\\
    (1) $\De$ has everywhere finite height;\\
    (2) $\ep_j=+1$ for all $j=1, \ldots, \mf$;\\
    (3) any point in $\partial^{\mathrm{top}}\De\cap\ell_{\la_j}$ for
      $j=1, \ldots, \mf$ satisfies either the hidden or fake condition 
      (and is referred to as either a \emph{hidden corner} or a \emph{fake corner}, respectively);\\
    (4) all other corners satisfy the Delzant condition, and are known
      as \emph{Delzant corners}. 
  The set of primitive semitoric polygons is denoted by $\pstpolyg$.
 
 For $\mf\in\Z_{\geq0}$ let
 $ G_\mf = \{-1, +1\}^\mf $ and $ \mathcal{G} = \{\,T^k \mid k\in\Z\,\}$
 where $T$ is as in Equation \eqref{eqn_T}. For $\la\in\R$ and $k\in\Z$ let 
 $t_{\ell_\la}^k\colon \R^2\to\R^2$ denote the map which acts as the identity on the left of the
 line $\ell_\la$ and acts as $T^k$ relative to an origin 
 placed arbitrarily on the line $\ell_\la$ to the right of $\ell_\la$.  Now for 
 $\vec{u} = (u_1, \ldots, u_\mf)\in\{-1, 0, 1\}^\mf$ and
 $\vec{\la} = (\la_1, \ldots, \la_\mf) \in\R^\mf$ define
 $t^{\vec{u}}_{\vec{\la}}\colon \R^2\to\R^2$ by
 \[ t^{\vec{u}}_{\vec{\la}} = t^{u_1}_{\ell_{\la_1}} \circ \ldots \circ t^{u_\mf}_{\ell_{\la_\mf}}.\]
 We define the action of an element of 
 $G_\mf\times\mathcal{G}$ on a labeled weighted polygon by
 $$
  \big((\ep_j')_{j=1}^\mf, T^k \big) \cdot \big(\De, (\ell_{\la_j}, \ep_j, k_j)_{j=1}^\mf\big) 
  = \big( t^{\vec{u}}_{\vec{\la}} \circ T^k (\De), (\ell_{\la_j}, \ep_j' \ep_j, k + k_j)_{j=1}^\mf\big)
 $$
 where $\vec{\la} = (\la_1, \ldots, \la_\mf)$ and 
 $
  \vec{u} = \left( \frac{\ep_j - \ep_j\ep_j'}{2}\right)_{j=1}^\mf.
 $
 This action may not preserve the convexity of $\De$
 but it is shown in~\cite[Lemma 4.2]{PeVNconstruct2011} that the
 orbit of a primitive semitoric polygon consists only of elements of $\lwpolyg$.
  \begin{definition}[\cite{PeVNconstruct2011}]
   A \emph{semitoric polygon} is the orbit under $G_\mf\times\mathcal{G}$
   of a primitive semitoric polygon.
  \end{definition}
  The collection of semitoric polygons
  is denoted by $\stpolyg = (G_\mf\times\mathcal{G})\cdot\pstpolyg$. 
  The orbit of 
 $\De_w = \big(\De, (\ell_{\la_j}, \ep_j, k_j)_{j=1}^\mf\big)\in\pstpolyg$
 is given by
 \[
  [\De_w] =  \{\, \big( t^{\vec{u}}_{\vec{\la}} \circ T^k (\De), (\ell_{\la_j}, 1 - 2u_j, k + k_j)_{j=1}^\mf\big) \,|\, \vec{u}\in\{0,1\}^\mf, k\in\Z\,\}. 
 \]
 The corners of any element of $[\De]$ are identified as hidden, fake,
 or Delzant similar to the case of the primitive semitoric polygon.

 \subsubsection{Volume invariant}
 For each $j=1, \ldots, \mf$ we let $h_j$
 denote the height of the image of the $j^{\mathrm{th}}$ focus-focus point from the
 bottom of the semitoric polygon.  Formally, this amounts to 
 $h_1, \ldots, h_\mf\in\R$ satisfying
 $ 0<h_j<\mathrm{length}(\pi_2(\De\cap\ell_{\la_j}))$
 for each $j=1, \ldots, \mf$.

 \subsubsection{Classification}
 Semitoric systems are classified by the invariants we have just reviewed.
 That is, the complete invariant of a semitoric system is an integer $\mf$,
 $\mf$ Taylor series, a collection of $\mf$ real numbers, and a labeled
 weighed semitoric polygon.  A single element of this orbit is shown in Figure~\ref{fig_primitive}.
 The complete invariant is an infinite
 family of such labeled weighted polygons, formed by a countably infinite
 number of subfamilies of size $2^\mf$ each parameterized by 
 $\vec{\ep} \in\{-1, +1\}^\mf$ (Figure~\ref{fig_familyofpoly}).   
 
 \begin{definition}[\cite{PeVNconstruct2011}] \label{stlist}
 A \emph{semitoric list of ingredients} is given by:
  \begin{enumerate}
   \item \emph{the number of focus-focus singularities invariant:}
    $\mf\in\Z_{\geq0}$;
   \item \emph{the Taylor series invariant:}
    a collection of $\mf$ elements of $\Rxyz$;
   \item \emph{the affine and twisting index invariants:}
    a semitoric polygon of complexity $\mf$, 
    the $(G_\mf\times\mathcal{G})$\--orbit of some 
    $\De_w = (\De, (\ell_{\la_j}, \ep_j, k_j)_{j=1}^\mf)\in\pstpolyg$;
   \item \emph{the volume invariant:}
    a collection of real numbers $h_1, \ldots, h_\mf\in\R$
    such that $0<h_j<\mathrm{length}(\pi_2(\De\cap\ell_{\la_j}))$
    for each $j=1, \ldots, \mf$.
  \end{enumerate}
  \end{definition}
  
  Let $\ingred$ denote the collection of all semitoric lists of ingredients.
  In~\cite{PeVNconstruct2011} the authors prove that
  semitoric manifolds modulo isomorphisms are classified by semitoric lists of ingredients,
  that is,
   \begin{align}
    \Phi\colon&  \mhamST \to\ingred\label{eqn_Phi}\\
    &(M, \om, (J,H)) \mapsto \big(\mf, ((S_j)^\infty)_{j=1}^\mf, [\De_w], (h_j)_{j=1}^\infty\big)\nonumber
   \end{align}
  is a bijection.

 \begin{figure}[ht]
 \centering
  \includegraphics[height=220pt]{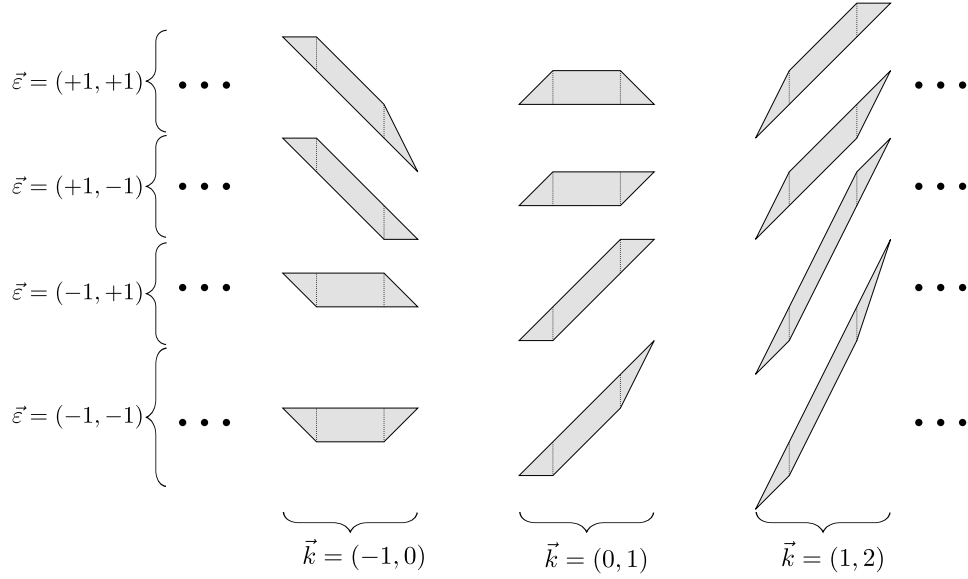}
  \caption{Complete invariant of a semitoric system.}
  \label{fig_familyofpoly}
 \end{figure}

\section{Symplectic \tex{$\T^n$}\--capacities}\label{sec_Tncap}

In this section we construct a symplectic $\T^n$\--capacity on the
space of symplectic toric manifolds. 
Recall $\mhamT$ is the moduli space of $2n$\--dimensional symplectic
toric manifolds up to $\T^n$\--equivariant symplectomorphisms which preserve the 
moment map.
In~\cite{FiPe2014, Pe2006, Pe2007, PeSc2008} the authors study the \emph{toric 
optimal density function} 
$\Om\colon \mhamT\to(0, 1],$
which assigns to each symplectic toric manifold the
fraction of that manifold which can be filled by equivariantly
embedded disjoint open balls.  This function is not a capacity
because it is not monotonic or conformal. Next we study a
modified version of this function which is a capacity.

For  $M\in\sympplain^{2n,\T^n}$ by a
\emph{$\T^n$\--equivariantly embedded ball} we mean the image $\phi(\mathrm{B}^{2n}(r))$
of a symplectic $\T^n$\--embedding $\phi\colon \mathrm{B}^{2n}(r)\immtn M$ for some $r>0$.
A \emph{toric 
ball packing} of $M$~\cite{Pe2006} is a disjoint union
$ 
 P = \bigsqcup_{\al\in\mathcal{A}} B_\al
$
where $B_\al\subset M$ is a symplecticly and $\T^n$\--equivariantly
embedded ball in $M$ for each $\al\in\mathcal{A}$, where $\mathcal{A}$
is some index set.  That is,
for each $\al\in\mathcal{A}$ there exists some $r_\al>0$ and
some symplectic $\T^n$\--embedding $\phi_\al \colon  \mathrm{B}^{2n} (r_\al)\immtn M$ 
such that 
\[\phi_\al (\mathrm{B}^{2n} (r_\al)) = B_\al.\]
An example is shown in Figure~\ref{fig_ballpackingdef}. 
Recall the toric packing capacity $\toricpack:\sympT\to[0,\infty]$
defined in Equation~\eqref{eqn_toricpack}.
In the following for $M\in\sympplain^{2n,\T^n}$ let
$\cBnn(M)$ be defined by first lifting the action of $\T^n$ on $M$
to an action of $\R^n$ and applying the usual $\cBnn$ to the resulting
symplectic $\R^n$\--manifold.
\begin{figure}[ht]
 \centering
 \includegraphics[height = 80pt]{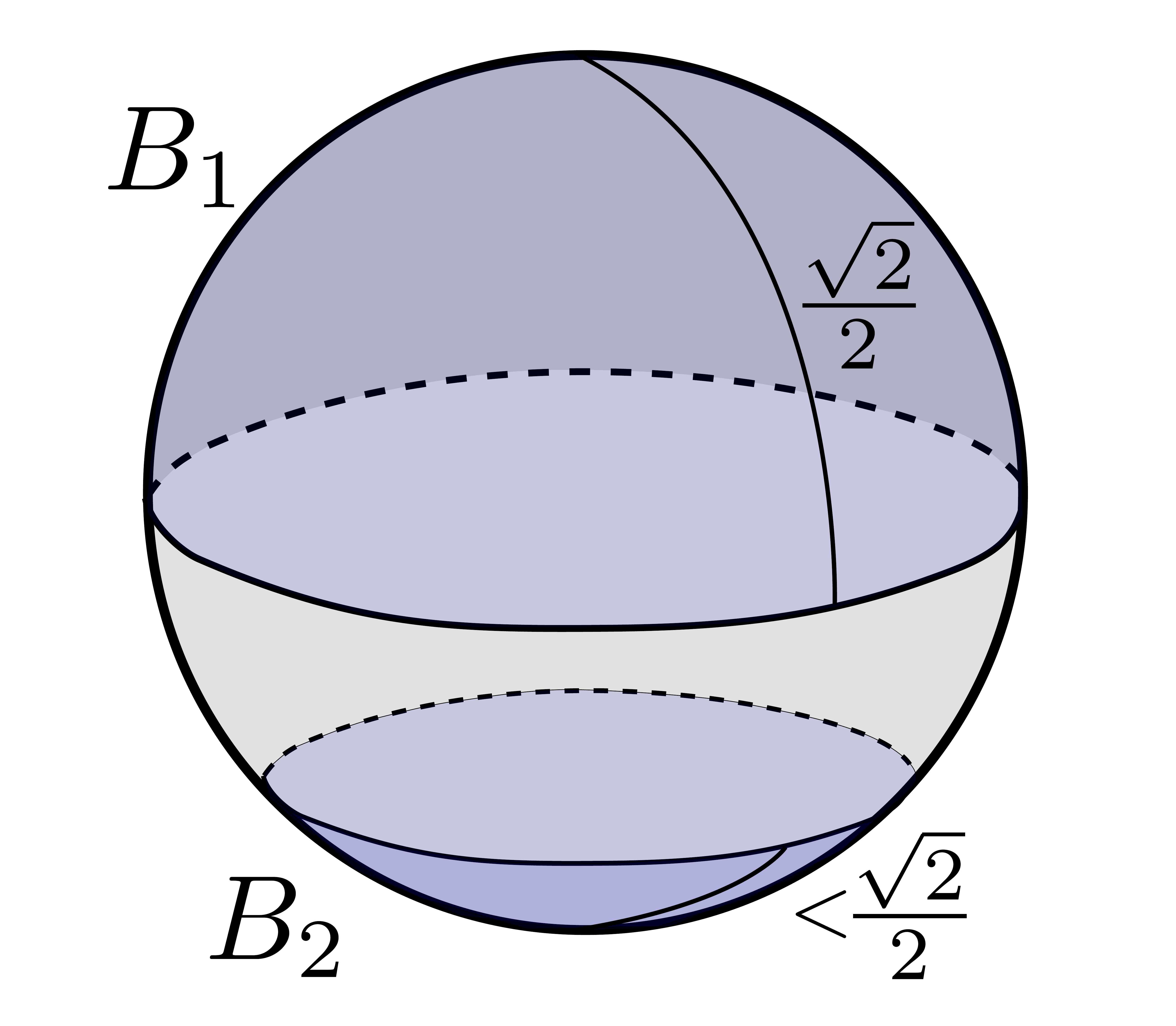}
 \caption{Toric ball packing of $S^2$ by symplectic $\T^2$\--disks.}
  \label{fig_ballpackingdef}
\end{figure}

 \begin{lemma}\label{lem_packingandradius}
  Let $M\in\sympT$, $N\in\sympplain^{2n, \T^n}$ be such that the $\T^n$\--action
  on $N$ has $\ell\in\Z_{\geq0}$  fixed points.  If there is
  a symplectic $\T^n$\--embedding $M\immtn N$ then
  $
   \toricpack (M) \leq \ell^{\nicefrac{1}{2n}} \cBnn(N).
  $
 \end{lemma}

 \begin{proof}
  Since the center of $\mathrm{B}^{2n}(r)$, $r>0$, is a fixed point of the $\T^n$\--action
  we see that the maximal number of such balls that can be
  simultaneously equivariantly embedded with disjoint images
  into $M$ is the Euler characteristic $\chi(M)$ of $M$,
  which is the number of fixed points of the $\T^n$\--action
  on $M$.  Each of these balls has radius at most $\cBnn(M)$.  For $r>0$
  we have that $\vol(\mathrm{B}^{2n}(r)) = r^{2n} \vol(\mathrm{B}^{2n})$.  Therefore
 $$
   (\toricpack (M))^{2n}\vol(\mathrm{B}^{2n})  \leq \chi (M) \vol\big(\mathrm{B}^{2n}(\cBnn(M))\big)
                    = \chi(M) \big( \cBnn (M)\big)^{2n}\vol(\mathrm{B}^{2n}).
$$
  Since $\T^n$\--embeddings send fixed points to fixed points and $M\immtn N$
  we know that $\chi(M) \leq \ell$.  Furthermore, since $M\immtn N$ and $\cBnn$ is
  a symplectic $\T^n$\--capacity by Proposition~\ref{propintro_equivgromov}
  we have that $\cBnn(M) \leq \cBnn (N)$. Hence  
  $\toricpack (M)\leq \ell^{\nicefrac{1}{2n}} \cBnn(N)$.
 \end{proof}

\begin{prop}\label{prop_toricpack}
 The toric packing capacity is a symplectic $\T^n$\--capacity
 on $\sympT$.
\end{prop}
 
\begin{proof} 
 Let $M\in\sympT$ with $\chi(M)\in\Z_{\geq0}$ fixed points and fix any 
 ordering of these points.
 Notice that
 $ \toricpack(M)$ is the supremum of 
 \[\{\, \|\vec{r}\|_{2n} \mid  \vec{r} \in\R^{\chi(M)},\,\, P_M (\vec{r})\subset M\textrm{ is a toric packing}\,\}\]
 where $\vec{r}=(r_1, \ldots, r_{\chi(M)})\in\R^{\chi(M)}$,
 \[
  \|\vec{r}\|_{2n} = \left(\sum_{j=1}^{\chi(M)} r_j^{2n}\right)^{\nicefrac{1}{2n}}
 \]
 is the standard $\ell^{2n}$\--norm, and
 $P_M (\vec{r})\subset M$ is the toric ball packing
 of $M$ in which $\mathrm{B}^{2n}(r_j)$ is embedded at the $j^{\mathrm{th}}$ fixed point of
 $M$ for $j=1, \ldots, \chi(M)$. 
 Suppose that $\rho\colon  \mathrm{B}^{2n}(r)\immtn M$ is a symplectic $\T^n$\--embedding into
 $(M, \om, \phi)$ for some $r>0$.
 Then for any $\la\in\R\setminus\{0\}$ the map $\rho_\la\colon  \mathrm{B}^{2n}(\abs{\la} r)\immtn M$
 given by 
 \[\rho_\la (z) = \rho (\nicefrac{z}{\abs{\la}})\]
 is a symplectic
 $\T^n$\--embedding into $(M, \la\om, \phi)$.  Thus if
 $P_M (\vec{r})$ is a toric packing of $(M, \om, \phi)$
 then $P_M (\abs{\la}r_1, \ldots, \abs{\la}r_{\chi(M)})$ is a toric
 ball packing of $(M, \la\om, \phi)$ for any $\la\in\R\setminus\{0\}$.  This and the fact
 that $\|\la r\|_{2n} = \abs{\la}\|r\|_{2n}$ for all $r\in\R^{\chi(M)}$ and $\la\in\R$
 imply that $\toricpack$ is conformal. Now suppose that $M, M'\in\sympT$ and $\rho\colon M\immtn M'$.  If $P\subset M$
 is a toric ball packing of $M$ then $\rho(P)\subset M'$ is a
 toric ball packing of $M'$ of the same volume so $\toricpack(M) \leq \toricpack(M')$ and
 we see that $\toricpack$ is monotonic. Finally, suppose that there is a symplectic 
 $\T^n$\--embedding $M\immtn \mathrm{Z}^{2n}$.  Then,
 since  $\mathrm{Z}^{2n}$ has only one point fixed by the $\T^n$\--action 
 and recalling that  $\cBnn(\mathrm{Z}^{2n})=1$, it follows from Lemma~\ref{lem_packingandradius} that 
 $$\toricpack(M) \leq (1)^{\nicefrac{1}{2n}} \cBnn(\mathrm{Z}^{2n}) = 1.$$ 

 Finally, suppose that $\rho\colon \mathrm{B}^{2n} \immtn M$ is a symplectic $\T^n$\--embedding. Then 
 $P = \rho(\mathrm{B}^{2n})\subset M$ is
 a toric ball packing of $M$ and thus 
 $$
  \toricpack(M) \geq \left(\frac{\vol(P)}{\vol(\mathrm{B}^{2n})} \right)^{\nicefrac{1}{2n}} = 1. 
$$
 Hence $\toricpack$ is tame.
 \end{proof}

 \begin{example}
  Let $M\in\sympT$.  In~\cite{Pe2007} it is shown that there exists a $\Z$\--valued function
  $\mathrm{Emb}_M\colon \R_{\geq0}\to[0,n!\chi(M)]$ such that the homotopy type of the space of symplectic
  $\T^n$\--embeddings from $\mathrm{B}^{2n}(r)$ into $M$ is given by the disjoint union of $\mathrm{Emb}_M(r)$
  copies of $\T^n$.  Thus, for each $r\in\R_{\geq0}$ we may define a symplectic 
  $\T^n$\--capacity $\mathcal{E}_r$ on $\sympT$ given by
  \begin{align*}
   \mathcal{E}_r \colon  \sympT &\to [0,\infty]\\
                  (M,\om,\phi) &\mapsto (\mathrm{vol}(M))^{\frac{1}{n}} \mathrm{Emb}_M ((\mathrm{vol}(M))^{\frac{1}{n}} r).
  \end{align*}
  Since $\mathrm{Emb}_M$ is invariant up to $\T^n$\--equivariant
  symplectomorphisms~\cite{Pe2007} and symplectic embeddings in $\sympT$ are automatically symplectomorphisms
  we see that $\mathcal{E}_r$ is monotonic and it is an exercise to check that it is conformal.
  It is tame because the space of symplectic $\T^n$\--embeddings of $\mathrm{B}^{2n}$ into $\mathrm{Z}^{2n}$ is homotopic
  to $n!$ disjoint copies of $\T^n$.
 \end{example}

\section{Symplectic \tex{$(S^1\times \R)$}\--capacities} \label{sec_symplSRcap}
  
 In this section we construct a symplectic $(S^1\times \R)$\--capacity
 on the space of semitoric manifolds.
  Let $(M,\om, F=(J,H))$ be a simple semitoric manifold with $\mf$
  focus-focus singular points and let $\{\la_j\}_{j=1}^\mf\subset\R$ be the image
  under $J$ of these points ordered so that $\la_1<\la_2<\ldots<\la_\mf$.
  Let $(\la_j, y_j)$ be the image under $F$ of the $j^{\textrm{th}}$ focus-focus
  singular point and 
  for $\ep\in\{\pm 1\}$ let $\ell^{\ep}_{\la_j}$ be those $(\la_j,y)\in\ell_{\la_j}$ 
  such that $\ep y > \ep y_j$.
  Let $\ell^{\vec{\ep}}=\ell^{\ep}_{\la_1}\cup \ldots \cup \ell^{\ep_\mf}_{\la_\mf}$.
  A homeomorphism 
  $$f\colon F(M)\to f(F(M))\subset \R^2$$
  is a \emph{straightening map for $M$}~\cite{VN2007} if for some
  choice of $\vec{\ep}\in\{\pm1\}^{\mf}$ we have the following:
 $f|_{F(M)\setminus\ell^{\vec{\ep}}}$ is a diffeomorphism
    onto its image;
  $f|_{F(M)\setminus\ell^{\vec{\ep}}}$ is affine with respect
    to the affine structure $F(M)$ inherits from action-angle coordinates
    on $M$ and the affine structure $f(F(M))$ inherits as a subset of $\R^2$;
   $f$ preserves $J$, i.e. $f(x,y) = (x, f^{(2)}(x,y))$;
   $f|_{F(M)\setminus\ell^{\vec{\ep}}}$ extends to a smooth multi-valued map
    from $F(M)$ to $\R^2$ such that for any $c=(x_0, y_0)\in\ell^{\vec{\ep}}$ we have
    $$
     \lim_{\substack{(x,y)\to c\\x<x_0}}\mathrm{d}f(x,y) = T \lim_{\substack{(x,y)\to c\\x>x_0}}\mathrm{d}f(x,y);
     $$ 
       and the image of $f$ is a rational convex polygon.
       Recall that $T$ is the matrix given in Equation~\eqref{eqn_T}.
  We say $f$ is associated to $\vec{\ep}$.

  Let $\mathfrak{T}\subset\mathrm{AGL}_2(\Z)$ be the subgroup including
 powers of $T$ composed with vertical translations.  It was proved in~\cite{VN2007} 
 that a semitoric system $(M,\om,F)$  has a straightening map
  $f\colon M\to\R^2$ associated to each $\vec{\ep}\in\{\pm1\}^{\mf}$, unique up to left composition
  with an element of $\mathfrak{T}$.   Define 
  \begin{equation}\label{eqn_FM}
   \mathcal{F}_M=\{\,f \circ F \mid f \textrm{ is a straightening map for $M$}\,\}.
  \end{equation}
  If $V_a\colon \R^2\to\R^2$ denotes vertical translation by $a\in\R$, then
  \[
   \{\,\widetilde{F}(M)\mid \widetilde{F}\in\mathcal{F}_M\,\} = \{\,V_a(\De)\subset\R^2\mid\De\textrm{ is associated to $M$ and }a\in\R\,\}
  \]
  where a polygon is associated to $M$ if it is an element of the affine invariant of $M$.
  Up to vertical translations the set $\mathcal{F}_M$ is the orbit of a single non-unique 
  function under the action of $G_\mf\times\mathcal{G}$.
  If $\widetilde{F}\in\mathcal{F}_M$ then there exists some
  $\vec{\ep}\in\{-1, +1\}^\mf$ such that
  $\widetilde{F}|_{M^{\vec{\ep}}}\colon M^{\vec{\ep}}\to\R^2$
  is a momentum map for a $\T^2$\--action
  $\phi_{\widetilde{F}}\colon \T^2\times M^{\vec{\ep}}\to M^{\vec{\ep}}$
  where $M^{\vec{\ep}} = M\setminus F^{-1}(\ell^{\vec{\ep}})$.  
    
 \begin{cor}\label{cor_mepsilon}
  The manifold 
  $M^{\vec{\ep}}$ has on
  it a momentum map for a Hamiltonian $\T^2$\--action unique up to $\mathcal{G}$.
  Thus $M^{\vec{\ep}}\in\sympplain^{4, \T^2}$ and the given $\T^2$\--action
  is unique up to composing the associated momentum map with an element
  of $\mathcal{G}$.  
 \end{cor}
 
 We call such actions of $\T^2$ on $M^{\vec{\ep}}$ \emph{induced actions of $\T^2$}.
 Given any $\rho\colon N\to M$ with $\rho (N) \subset M^{\vec{\ep}}$
 define $\rho_{\vec{\ep}}:N\to M^{\vec{\ep}}$
 by $\rho_{\vec{\ep}}(p) = \rho(p)$ for $p\in N$.
 
 \begin{definition}\label{def_stemb}
  Let $(M,\om,F)$ be a semitoric manifold and let 
  $(N, \om_N, \phi)\in\sympplain^{4, \T^2}$.
  A symplectic embedding $\rho\colon N\imm M$ is a 
  \emph{semitoric embedding} if there exists
  $\vec{\ep}\in\{\pm1\}^\mf$ and an induced action 
  $\phi_{\vec{\ep}}:\T^2\times M_{\vec{\ep}}\to M_{\vec{\ep}}$ such that $\rho(N)\subset M^{\vec{\ep}}$ and 
  $\rho_{\vec{\ep}}\colon (N, \om_N, \phi) \adjustedarrow{\T^2} (M^{\vec{\ep}}, \om, \phi_{\vec{\ep}})$
  is a symplectic $\T^2$\--embedding.
 \end{definition}

  Let $(M, \om, F)$ be a semitoric manifold. 
   A \emph{semitoric 
  ball packing} of $M$ is a disjoint union
  $ P = \bigsqcup_{\al\in\mathcal{A}} B_\al$
  where  $B_\al\subset M$ is a semitoricly embedded ball
  in $M$.  
 The \emph{semitoric packing capacity} 
 $
  \semitoricpack\colon \sympST\to[0, \infty]
 $
 is given by
 \[
  \semitoricpack (M) = \left( \frac{\sup\{\,\mathrm{vol}(P)\mid P\subset M\textrm{ is a semitoric ball packing of $M$} \,\}}{\vol (\mathrm{B}^{4})}\right)^{\frac{1}{4}}.
 \]
  In order to show that $\semitoricpack$ is a $(S^1\times\R)$\--capacity we need the following lemmas.
  \begin{lemma}\label{lem_hamvectfield}
  For $i=1,2$ let $(M_i, \om_i)$ be a symplectic manifold, let
  $f_i\colon M_i\to\R$ be a function, and let $\mathcal{X}_{f_i}$
  denote the Hamiltonian vector field of $f_i$ on $M_i$.
  If $\rho\colon M_1\to M_2$ is a symplectomorphism such that
  $\rho_* \mathcal{X}_{f_1}=\mathcal{X}_{f_2}$ then $f_1-\rho^* f_2\colon M_1\to\R$ is constant.
 \end{lemma}
 
 \begin{proof}
  Notice that
 \begin{align*}
   \mathrm{d}(\rho^* f_2) &= \rho^*(\mathrm{d}f_2)
                          = \rho^*(\io_{\mathcal{X}_{f_2}}\om_2)
                          = \rho^*(\io_{\rho_* \mathcal{X}_{f_1}}\om_2)\\
                         &= \om_2 (\rho_* \mathcal{X}_{f_1}, \rho_* (\cdot))
                          = (\rho^*\om_2) (\mathcal{X}_{f_1}, \cdot)
                          = \io_{\mathcal{X}_{f_1}}\om_1
                          = \mathrm{d}f_1,
  \end{align*}
 thus $f_1$ and $\rho^*f_2$ differ by a constant.
 \end{proof}

 \begin{lemma}\label{lem_equivsymplecto}
  Let $(M_i, \om_i, F_i=(J_i, H_i))$ be semitoric manifolds for $i=1,2$.  
  If $\rho\colon M_1\immTR M_2$ is a symplectic $(S^1\times\R)$\--embedding with respect to
  the Hamiltonian flow action on each system, then
 $$
   \rho^* J_2 = e J_1 + c_J\quad\text{ and }
   \quad\rho^* H_2 = a J_1 + b H_1 + c_H$$
   for some $e\in\{\pm1\}$ and $a, b, c_J, c_H \in \R$ such that $b\neq 0$.
 \end{lemma}
 
 \begin{proof}
 Since $\rho$ is $S^1\times\R$\--equivariant there exists $\La\in\mathrm{Aut}(S^1\times\R)$
 such that $\rho(\phi(g,m_1)) = \phi(\La(g), \rho(m_1))$
 for all $g\in S^1\times\R$ and $m_1\in M_1$.
 Associate $ S^1\times\R$ with $\R/\Z\times\R$ and give it coordinates $(x,y)\in\R^2$.
 Then $\La\in\mathrm{Aut}( S^1\times\R)$ and $\La$ continuous
 means that $\La$ descends from a linear invertible map from $\R^2$ to
 itself, which we will also denote $\La\in\mathrm{GL}_2(\R)$.  Write $\La = (\La_{ij})$ for 
 $\La_{ij}\in\R$ and $i,j\in\{1,2\}$. The automorphism $\La$ sends the identity to
 itself so
 $
  \La\vect{n}{0}\in\Z\times\{0\}
 $
 for all choices of $n\in\Z$.  This implies that $\La_{11} \in\Z$ and $\La_{21}=0$.
 Since $\La$ is invertible and $\La^{-1}\in\mathrm{Aut}( S^1\times\R)$ we
 see that $(\La_{11})^{-1}\in\Z$ and so $\La_{11} = \pm 1$.  Since
 $\La$ is invertible and upper triangular we know that $\La_{22}\neq 0$.

 For a function $f\colon M_i\to\R$ let $\mathcal{X}_f$ denote the associated 
 Hamiltonian vector field on $M_i$, $i=1,2$.  Also, for
 $v\in\mathfrak{g}=\mathrm{Lie}(S^1\times\R)$, thought of as
 the tangent space to the identity, let $v_{M_i}$ denote
 the vector field on $M_i$ generated by $v$ by the group action.
 Endow $\mathfrak{g}$ with the coordinates $(\al,\be)$ so that the exponential
 map will send $( \al, \be )\in\mathfrak{g}$ to
 $(\al, \be) \in \R/\Z\times\R$.
 Now notice that $\mathcal{X}_{J_1}=(1,0)_{M_1}$ and $\mathcal{X}_{H_1}=(0,1 )_{M_1}$.
 
 For $m_i\in M_i$, $i=1,2$, such that $\rho(m_1)=m_2$ we have
 \[
  \rho_* \mathcal{X}_{J_1}(m_2) = \frac{\mathrm{d}}{\mathrm{d}t}\bigg|_{t=0}\Big( \rho\big(\phi( (t, 0), m_1)\big) \Big)
                      = \frac{\mathrm{d}}{\mathrm{d}t}\bigg|_{t=0}\Big( \phi( \La[(t, 0)], m_2) \Big)
                      = \big(\mathrm{T}\La ( 1,0)\big)_{M_2} (m_2)
 \]
 Notice that $\mathrm{T}_{( 1, 0 )} = ( \La_{11}, 0 ) \in \mathfrak{g}$.
 Then
 $
  \rho_* \mathcal{X}_{J_1} = \big(\mathrm{T}\La( 1, 0 )\big)_{M_2} = \La_{11} ( 1 , 0 )_{M_2} = \La_{11}\mathcal{X}_{J_2}.
 $
 Similarly we see that
 $
  \rho_* \mathcal{X}_{H_1} = \La_{12}\mathcal{X}_{J_2} + \La_{22} \mathcal{X}_{H_2}.
 $
 By Lemma~\ref{lem_hamvectfield} this implies that
 $$
  \rho^* J_2 = \frac{1}{\La_{11}} J_1+c_J\qquad \text{and}\qquad\rho^* H_2 = \frac{-\La_{12}}{\La_{11}\La_{22}} J_1 + \frac{1}{\La_{22}} H_1 + c_H$$
 for some $c_J, c_H\in\R$.
 Recalling that $\La_{11}\in\{\pm 1\}$ and $\La_{11}, \La_{22}\neq 0$ take
 $e=(\La_{11})^{-1}$, $a = \frac{-\La_{12}}{\La_{11}\La_{22}}$, and $b = (\La_{22})^{-1}$
 to complete the proof.
 \end{proof}
 
\begin{prop}\label{prop_stpack}
 The semitoric packing capacity, $\semitoricpack$, is a 
 symplectic $( S^1\times\R)$\--capacity on $\sympST$.
\end{prop}

 \begin{proof}
  The proof that $\semitoricpack$ is conformal and non-trivial
  is analogous to the proof of Proposition~\ref{prop_toricpack},
  so we must only show that $\semitoricpack$ is monotonic. 
  Let $(M_i, \om_i, F_i)$ be semitoric for $i=1,2$
  and suppose $\phi\colon  M_1\immTR M_2$
  is a symplectic $(S^1\times\R)$\--embedding.
  Recall that action-angle coordinates are local Darboux charts
  in which the flow of the Hamiltonian vector fields are linear.
  Since $\phi$ is symplectic, $(S^1\times\R)$\--equivariant, and
  $\phi^*(F_2) = A\circ F_1$ where $A\colon \R^2\to\R^2$ is 
  affine (Lemma \ref{lem_equivsymplecto}),
  this means that $\phi$ sends action-angle coordinates to action-angle coordinates.
  Since semitoric embeddings are those which respect the action-angle coordinates,
  given any semitoric embedding $\rho\colon \mathrm{B}^{2n}(r)\imm M_1$ the map
  $\phi\circ\rho\colon \mathrm{B}^{2n}(r)\imm M_2$ is a semitoric embedding.
  It follows that $\semitoricpack(M_1)\leq\semitoricpack(M_2)$.
 \end{proof}

 Proposition~\ref{propintro_packingcaps} follows from
 Propositions~\ref{prop_toricpack} and~\ref{prop_stpack}.

\section{Continuity of symplectic \tex{$\T^n$}\--capacities}\label{sec_toricpackcont}  
 
  In this section we study the continuity of the symplectic $\T^n$\--capacity
  constructed in Section~\ref{sec_Tncap}.
  We will outline the procedure used in~\cite{PePRS2013} to construct 
  a natural metric on the moduli space of toric manifolds.
  Since $\Psi\colon  \mhamT\to\P_\toric$  is a bijection we can define a metric space
  structure on $\mhamT$ by defining a metric on $\P_\toric$ and pulling 
  it back via $\Psi$.  A natural metric on $\P_\toric$ is
  given by the volume of the symmetric difference.  For $A, B \subset \R^n$ let 
  $
   A \symdiff B = (A \setminus B) \cup (B \setminus A)
  $
  denote the symmetric difference and let $\la$ denote the
  Lebesgue measure on $\R^n$. For $\De_1, \De_2 \in \P_\toric$ define 
  $d_\P (\De_1, \De_2) = \la (\De_1 \symdiff \De_2)$.
  Now let $d_\toric = \Psi^* d_\P$. In~\cite{PePRS2013}
  the authors show that $(\mhamT, d_\toric)$ is a non-locally compact 
  non-complete metric space.

 The map 
 \[\mhamT\to\msympT\]
 given by $[(M,\om,\phi,\mu)]\mapsto[(M,\om,\phi)]$
 is a quotient map and thus we can endow $\msympT$ with
 the quotient topology.
 Since $\msympT$ is a quotient of $\sympT$ we can pull the topology up
 from $\msympT$ to $\sympT$ by declaring that a set in $\sympT$ is open
 if and only if it is the preimage of an open set from $\msympT$ under
 the natural projection.
 Two points in $\sympT$ are not separable if and only if
 they are $\T^n$\--equivariantly symplectomorphic.
 Thus a map $c\colon \sympT\to [0,\infty]$ which descends to a well-defined
 map $\phi$ on $\msympT$ is continuous if and only if the map
 \[\hat{c}\colon \mhamT\to [0,\infty]\]
 is continuous, where $\hat{c}$ is
 defined by the following commutative diagram:

 \begin{equation*}
   \begin{tikzcd}
   \hamT \ar{d}{}  \ar{r}{}           &  \sympT   \ar{d}{} \ar{r}{c}      & {[0,\infty]} \\
   \mhamT \ar{r}{} \ar[bend right=55, swap]{rru}{\hat{c}}          &  \msympT \ar{ur}{\phi} &   \\
   \end{tikzcd}
  \end{equation*}
 
 Next we define an operation on Delzant polytopes.  Let $n\in\Z_{>0}$.
 For $x_0\in\R^n$, $w_1, \ldots,  w_n\in\Z^n$, and $\varep>0$ define 
 \begin{equation}\label{eqn_halfplane}
  \mathcal{H}_{x_0}^\varep (w_1,\ldots,w_n) = \{\, x_0+\textstyle\sum\nolimits_{j} t_j w_j \mid t_1, \ldots, t_n\in\R_{\geq0}, \sum_j t_j\geq \varep\,\}.
 \end{equation}
 
 Suppose that $\De \in \P_\toric$ and $x_0\in\R^n$
 is a vertex of $\De$.  
 Let $u_i\in\Z^n$, $i=1,\ldots,n$, denote the primitive vectors along which
 the edges adjacent to $x_0$ are aligned.
 The \emph{$\varep$\--corner chop of $\De$ at $x_0$} is the 
 polygon $\De_{x_0}^\varep \in \P_\toric$ given
 by 
 $
  \De_{x_0}^\varep = \De \cap \mathcal{H}_{x_0}^{\varep} (u_1, \ldots, u_n)
 $
 where $\varep$ is
 sufficiently small so that $\De_{x_0}^\varep$ has exactly one more face than $\De$
 does as is shown in Figure~\ref{fig_epsilonchop}. 
 \begin{figure}[ht]
  \centering
  \includegraphics[height=50pt]{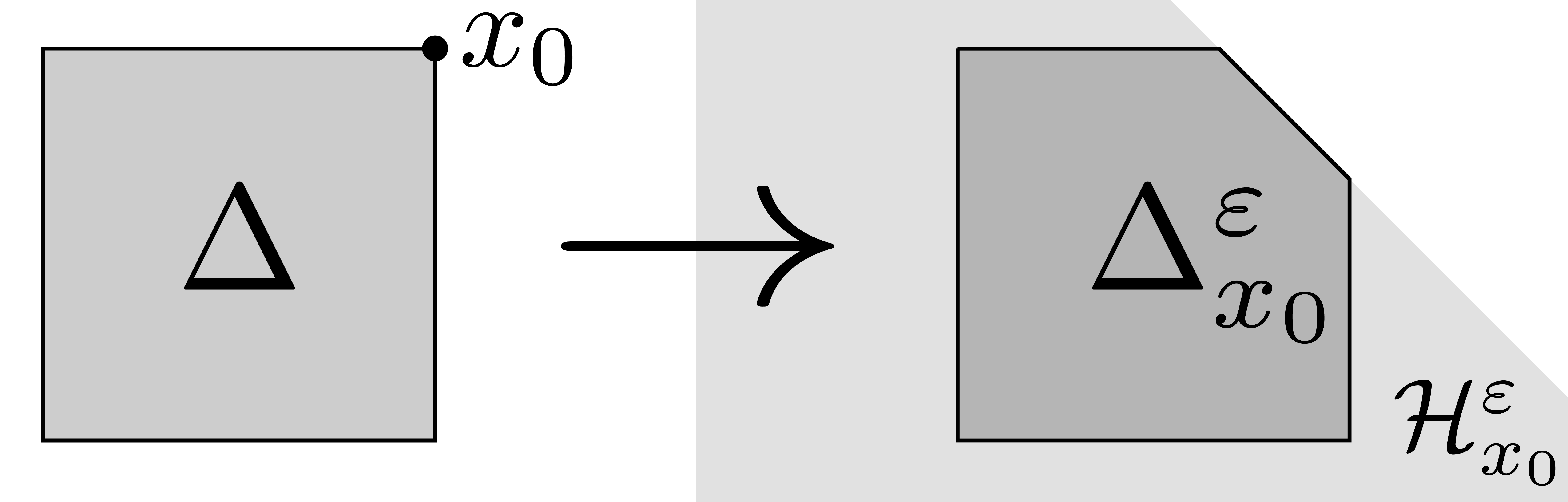}
  \caption{An $\varep$\--corner chop at a vertex
   $x_0$ of $\De$ for some $\varep>0$.}
  \label{fig_epsilonchop}
 \end{figure}
 One can check that if $\De\in\P_\toric$ then $\De_{x_0}^\varep\in\P_\toric$.
 Notice that $\lim_{\varep \to 0} d_\P (\De, \De_{x_0}^\varep) = 0$. This
 means that given any element of $\P_\toric$ with $N$ vertices, corner chopping
 can be used to produce other polygons which are close in $d_\P$ and all 
 polygons produced in this way will have more than $N$ vertices.
 Let $\P_\toric^N$ denote the set of Delzant polygons in $\R^n$ with exactly $N$ 
 vertices.  We will later need the following.
   
 \begin{prop}[\cite{FiPe2014}]\label{prop_simpleobs}
  Let $N\in\Z_{>0}$ and $\De\in\P_\toric^N$.  Any sufficiently small neighborhood 
  of $\De$ is a subset of $\cup_{(N'\geq N)} \P_\toric^{N'}$.
 \end{prop}
 
 \begin{figure}[ht]
  \centering
  \includegraphics[height=60pt]{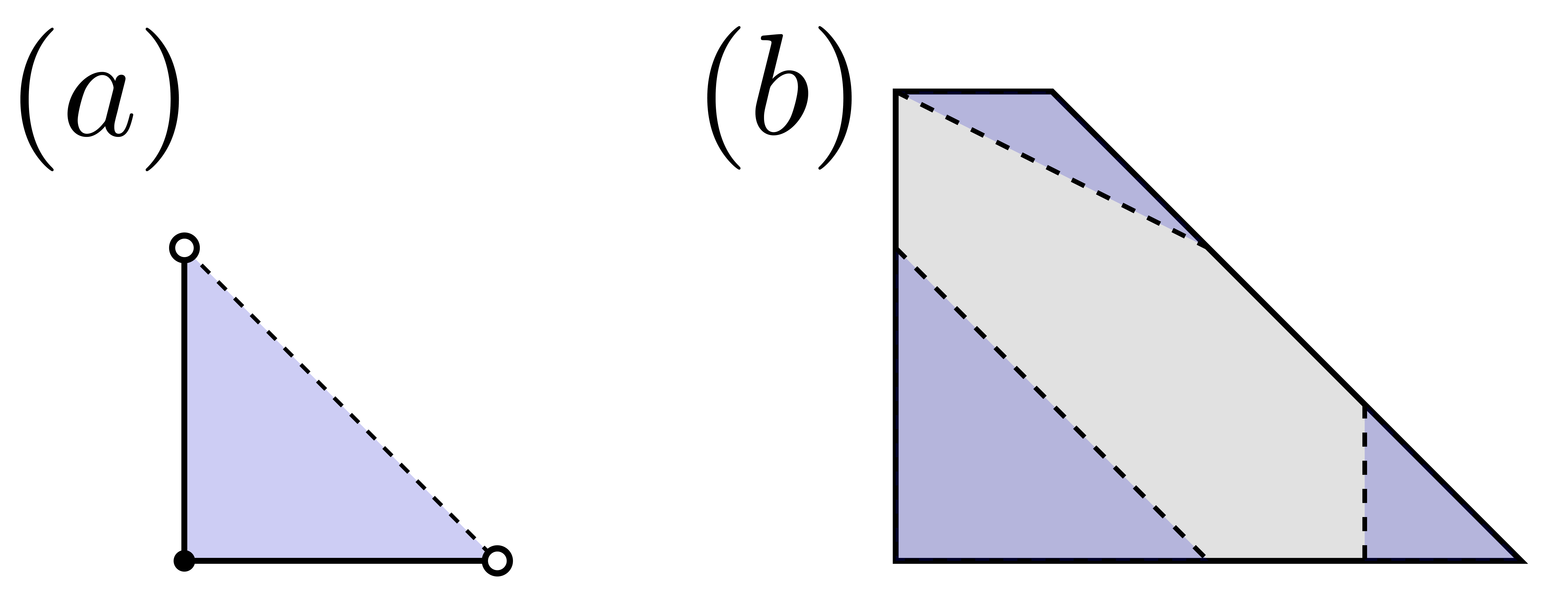}
  \caption{(a) An image of $\De(1)\subset\R^2$.
   (b) An image of an admissible, but not maximal, packing.}
  \label{fig_admpacking}
 \end{figure}
 
 We study ball packing problems about symplectic toric manifolds by instead
 studying packings of the associated Delzant polygon.
 Let $\De\in\P_\toric$ be a Delzant polytope. Let 
 $\mathrm{AGL}_n(\Z) = \mathrm{GL}_n (\Z) \ltimes \R^n$ denote the group
 of affine transformations in $\R^n$ with linear part in 
 $\mathrm{GL}_n(\Z)$. For $r>0$ let 
 $\De(r) = \mathrm{Conv}\{\,re_1, \ldots, re_n, 0\,\}\setminus 
 \mathrm{Conv}\{\,re_1, \ldots, re_n\,\}$
 where $\mathrm{Conv}(E)$ denotes the convex hull of the
 set $E\subset\R^n$ and $\{\,e_1, \ldots, e_n\,\}$ denote the
 standard basis vectors in $\R^n$.
 Following~\cite{Pe2006}, a subset $\Si$ of $\De$ is an \emph{admissible simplex of radius
 $r>0$ with center at a vertex $x_0$ of $\De$} if there exists some 
 $A\in\mathrm{AGL}_n(\Z)$ such that:\\
 (1) $A(\De(r^{\nicefrac{1}{2}})) = \Si$;\\ (2) $A(0) = x_0$; \\
 (3) $A$ takes the edges of $\De(r^{\nicefrac{1}{2}})$ meeting at the origin
 to the edges of $\De$ meeting at $x_0$.
 
 An \emph{admissible packing of $\De$} is a disjoint union
 $R = \bigsqcup_{\al\in \mathcal{A}} \Si_\al\subset\De$
 where each $\Si_\al$ is an admissible simplex for $\De$. 
 This is illustrated in Figure~\ref{fig_admpacking}.
 The half-plane $\mathcal{H}_{x_0}^\varep$ given
 in Equation \eqref{eqn_halfplane}
 is designed so that that an $\varep$\--corner chop on a Delzant
 polytope corresponds to the removal of an admissible
 simplex of radius $\varep$.
 
 The function $\Om\colon  \msympT \to (0,1]$ given by 
 \[ \Om (M) = \frac{\sup\{\,\vol(P) \mid P \textrm{ is a toric ball packing of }M\,\}}{\vol(M)},\]
 known as the \emph{optimal toric density function}, has been
 studied in~\cite{FiPe2014, Pe2006, PeSc2008}.
 In particular, in~\cite{FiPe2014} the first 
 and third authors of the present article studied the regions of continuity
 of $\Om$ and proved the $n=2$ case of Theorem~\ref{thmintro} part~\ref{thmintro_part_toric}.
 They stated the theorem in terms of $\Om$, while we state it
 in terms of $\toricpack$. 
 
 Let $\mathrm{vol}\colon \sympT \to \R$ denote the total symplectic 
 volume of a symplectic toric manifold and let 
 $\mathrm{vol}_\P \colon  \P_\toric \to \R$ denote 
 Euclidean volume function of a polytope in $\R^n$.
 Let $(\mathrm{B}^{2n}(r), \om_\mathrm{B}, \phi_\mathrm{B}, \mu_\mathrm{B})\in\hamT$ denote the standard
 ball of radius $r>0$ in $\C^n$ with the standard action of $\T^n$ and
 suppose that $(M,\om,\phi, \mu)\in\hamT$. Let 
 $\De_\mathrm{B} = \mu_\mathrm{B} (\mathrm{B}^{2n}(r))$ and $\De = \mu (M)$.
 Then, as shown in~\cite{Pe2006}, $\mathrm{vol} (M) = n! \pi^n \mathrm{vol}_\P (\De)$ and
 if $f\colon \mathrm{B}^{2n}(r) \immtn M$ is a 
 symplectic $\T^n$\--embedding then
 \[
  \mathrm{vol}(\mathrm{B}^{2n}(r)) = \mathrm{vol}(f(\mathrm{B}^{2n}(r))) = n!\pi^n \mathrm{vol}_\P (\mu\circ f(\mathrm{B}^{2n}(r))) = n!\pi^n \mathrm{vol}_\P (\De_\mathrm{B}).
 \]

 \begin{theorem}[\cite{Pe2006}]\label{thm_toricgeomandcombo}
  Let $(M,\om,\phi,\mu)\in \hamT$ and let $\De = \mu(M)$.   Suppose $\phi\colon \mathrm{B}^{2n}(r)\imm M$ is a symplectic 
    $\T^n$\--embedding for some $r>0$.  Then $\mu(\phi(\mathrm{B}^{2n}(r)))\subset\De$ is an admissible
    simplex of radius $r^2$.  Conversely, if $\Si\subset\De$ is an admissible
    simplex of radius $r^2$ then there exists a symplectic 
    $\T^n$\--embedding $\phi\colon \mathrm{B}^{2n}(r)\imm M$
    such that $\mu(\phi(\mathrm{B}^{2n}(r))) = \Si$.
   
   Moreover,  if $P$ is a toric ball packing of $M$, then $\mu(P)\subset\De$ is an admissible packing of $\De$. 
   Conversely, if $R$ is an admissible packing of $\De$ then there exists a
    toric ball packing $P$ of $M$ such that 
    $\mu(P) = R$.
 \end{theorem}

Since there is a toric ball
 packing $P$ of $M$ related to an admissible packing
 $R$ of $\De$ by $\mu(P)=R$,
 it follows that
 $
  \mathrm{vol}(P)= n!\pi^n\mathrm{vol}_\P(R).
 $
 To study packing of the manifold we will study packing of the polygon.
   Thus, we define
  $\pack_\toric \colon \P_\toric\to(0, \infty)$ by
  $$
   \pack_\toric (\De) = \sup\{\,\mathrm{vol}_\P(R)\mid R\textrm{ is an admissible packing of $\De$}\,\}.
  $$
  Suppose that $\De \in \P_\toric^N$ with vertices
    $v_1, \ldots, v_N\in\R^n$ and let
    $\pack_\toric^i (\De)$ be the supremum of $\mathrm{vol}_\P(R)$ 
    over all admissible packings $\mathcal R$ of $\De$ in which $v_i\notin \mathcal R$.

The following result generalizes~\cite[Theorem 7.1]{FiPe2014} to the case $n \geq 3$.

 \begin{theorem}\label{thm_toricpolycont}
  Fix $n\in\Z_{>0}$.  For $N\in\Z_{\geq1}$ and let $\P_\toric^N$ denote the set of 
  Delzant polygons in $\R^n$ with exactly $N$ 
  vertices. Then:
  \begin{enumerate}[font=\normalfont]
   \item\label{part_polythm1} $\pack_\toric$ is discontinuous at 
    each point in $\P_\toric$;
   \item\label{part_polythm2} the restriction 
    $\pack_\toric|_{\P_\toric^N}$ is continuous for each $N\geq 1$;
   \item\label{part_polythm3} if $\De \in \P_\toric^N$ then $\P_\toric^N$ is the largest
    neighborhood of $\De$ in $\P_\toric$ in which $\pack_\toric$ is continuous if and only if 
    $\pack_\toric^i (\De) < \pack_\toric (\De)$ 
    for all $1\leq i \leq N$.
  \end{enumerate}
 \end{theorem}

 \begin{proof}
  First we show \eqref{part_polythm1}.  Let $\De\in\P_\toric^N$ and
  for any small enough $\varep>0$ perform an $\varep$\--corner chop 
  (as in Section~\ref{sec_toricpackcont}) at each
  corner to produce $\De_\varep\in\P_\toric^{2N}$.
  Any admissible packing of $\De_\varep$ can have at most $2N$ simplices 
  and each simplex must have one side with length at most $\varep$ while 
  the other sides are universally bounded by the maximal side length 
  of $\De$.  The size of such simplices decreases to zero as $\varep$ does, so
  $
   \lim_{\varep\to0}\pack_\toric (\De_\varep) =0.
  $
Hence
  $$
   \lim_{\varep\to0}d_\P (\De, \De_\varep) = 0
  $$
  but 
  $$
   \lim_{\varep\to 0}\abs{\pack_\toric (\De) - \pack_\toric(\De_\varep)} = \pack_\toric(\De)>0,
 $$
 so $\pack_\toric$ is discontinuous at $\De$.
  
  Now we prepare to show part \eqref{part_polythm2}.
  For any $v_1, \ldots, v_n \in \Z^n$ let $[v_1, \ldots, v_n]$
  denote the $n\times n$ integer matrix with $i^{\textrm{th}}$ column given by $v_i$ 
  for $i=1, \ldots, n$.  Let $\eta\colon \mathrm{SL}_n(\Z)\to\mathrm{GL}_n(\R)$ given by 
  $$
   \eta([v_1, \ldots, v_n]) = \left[ \frac{v_1}{\abs{v_1}}, \ldots, \frac{v_n}{\abs{v_n}}\right] 
  $$
  take a nonsingular integer matrix to its \emph{column normalization}.
  Notice for any $A = [v_1, \ldots, v_n]\in\mathrm{SL}_n(\Z)$ that
  $$
   \det(A) = \abs{v_1}\cdots\abs{v_n}\cdot\det(\eta(A)).
  $$
  Suppose $\De\in\P_\toric$ is $n$\--dimensional.  In a neighborhood 
  around each vertex the polytope is described by a collection of vectors 
  $v_1, \ldots, v_n\in\Z^n$ with $\det(v_1, \ldots, v_n)=1$ along which the edges adjacent to this vertex 
  are directed. So, associated to any vertex of a Delzant 
  polytope, there is a matrix $A\in\mathrm{SL}_n(\Z)$ given by
  $
   A = [v_1, \ldots, v_n]
  $
  which is unique up to even permutations of its columns
  and thus, though $A$ is not 
  unique, the values determined by $\det(A)$ and $\det(\eta(A))$
  associated to a vertex are well-defined.
  Fix $\De\in\P_\toric^N$ and $\{\,\De_j\,\}_{j=1}^\infty \subset \P_\toric^N$ such that
  \begin{equation}\label{eqn_dPlimit}
   \lim_{j\to\infty} d_\P (\De, \De_j) = 0.
  \end{equation}
  For $j$ large enough for each vertex $V$ of $\De$ there must be a 
  corresponding vertex $V_j$ of $\De_j$ so that  $V_j\to V$
  as $j\to \infty$.  Let $A\in\mathrm{SL}_n(\Z)$ be a matrix corresponding to $V$ and let 
  $A_j\in\mathrm{SL}_n(\Z)$ be a matrix corresponding to $V_j$ for $j\in\Z$ 
  large enough.  In particular, convergence in $d_\P$, which 
  is convergence in $\mathrm{L}^1(\R^n)$, implies that locally these vertices must converge, so
  Equation~\eqref{eqn_dPlimit} implies that
  $$
   \lim_{j\to\infty} \abs{\det(\eta(A)) - \det(\eta(A_j))}=0.
  $$

  Now we are ready to prove \eqref{part_polythm2} by showing that 
  the collection of possible vertices of Delzant polytopes is discrete.
  Fix $\De\in\P_\toric^N$  with a vertex $V$ at the origin and let $\varep>0$.
  Choose $\de>0$ small enough so that if $\De'\in\P_\toric^N$ with
  a vertex $V'$ at the origin then $d_\P(\De, \De')<\de$ implies that 
  \begin{equation}\label{eqn_etaA}
   \abs{\det(\eta(A))-\det(\eta(A'))}<\varep,
  \end{equation}
  where $A\in\mathrm{SL}_n(\Z)$ is a matrix associated to $V$ and
  $A'\in\mathrm{SL}_n(\Z)$ is a matrix associated to $V'$.
  Suppose that $\varep < \det(\eta(A))$.  Now let 
  $A' = [w_1, \ldots, w_n]$ for $w_i\in\Z^n$, $i = 1, \ldots, n$.
  These are all nonzero integer vectors so $\abs{w_i}\geq 1$ 
  for $i = 1, \ldots, n$.  For each $i$ we have
  \[
   1 = \det(A') = \abs{w_1}\abs{w_2}\ldots\abs{w_n}\det(\eta(A')) \geq \abs{w_i} \det(\eta(A'))
  \]
  and so by Equation \eqref{eqn_etaA}
  $$
   \abs{w_i} \leq \frac{1}{\det(\eta(A'))} \leq \frac{1}{\det(\eta(A))-\varep}.
  $$
  Thus each $w_i\in\Z^n$ has length at most 
  $(\det(\eta(A))-\varep)^{-1}$,
  a value which does not 
  depend on $\De'$, and so to be within $\de$ of $\De$
  the vectors directing the edges coming out from the vertex $V'$ 
  of $\De'$ must be chosen from only finitely many options.
  This means the set of possible local neighborhoods of vertices is discrete.
  Thus, for small enough $\de>0$ we conclude
  that $d_\P (\De, \De')<\de$ implies that there
  exist open sets $U, U'\subset\R^n$ around the vertices $V$
  and $V'$ such that 
  \[\De\cap U = F_c(\De'\cap U')\]
  where $F_c\colon \R\to\R$ is a translation by some fixed $c\in\R^n$.
    Now, let $\De\in\P_\toric^N$ be any Delzant polytope in $\R^n$ with $N$ vertices.
  In a sufficiently small $d_\P$\--neighborhood 
  of $\De$ all polytopes must have the same angles at the
  finitely many vertices by 
  the argument above.  Thus they are all related to $\De$ by translating 
  its faces in a parallel way, which continuously changes $\pack_\toric$.
  This proves \eqref{part_polythm2} because $\pack_\toric$ is continuous
  on such families.
  
  Finally we show \eqref{part_polythm3}.  Let $\De\in\P_\toric^N$ 
  and assume that $\pack_\toric(\De) = \pack_\toric^i(\De)$
  for some $i\in\{1, \ldots, N\}$.  Then
  there is an optimal packing of $\De$ which avoids the $i^{\textrm{th}}$ 
  vertex.  For $\varep>0$ let $\De_\varep\in\P_\toric^{N+1}$ be the $\varep$\--corner
  chop of $\De$ at the $i^{\textrm{th}}$ vertex.  Since the optimal packing
  of $\De$ avoids the $i^\textrm{th}$ vertex, we see that
  $\lim_{\varep\to0}d_\P (\De, \De_\varep) = 0 $
  and
  $\lim_{\varep\to0}\pack_\toric (\De) = \pack_\toric (\De_\varep)$
  so there is a set larger than $\P_\toric^N$ on which 
  $\pack_\toric$ is continuous around $\De$.
  
  Conversely assume that $\De\in\P_\toric^N$ satisfies
  $\pack_\toric^i(\De)<\pack_\toric (\De)$
  for all $i =1, \ldots, n$.  By Proposition~\ref{prop_simpleobs} we 
  know that any small enough neighborhood of $\De$ only
  includes polytopes with $N$ vertices and polytopes with more than $N$ 
  vertices, which are produced from corner chops of $\De$. We must now only
  show that $\pack_\toric$ cannot be continuous on any neighborhood
  of $\De$ which includes any such polygons.  For $\varep>0$
  let $\De_\varep\in\P_\toric^{N+1}$ be the $\varep$\--corner chop of 
  $\De$ at the $i^\textrm{th}$ vertex.  Then
  $\lim_{\varep\to 0}\pack_\toric (\De_\varep) = \pack_\toric^i (\De) < \pack_\toric$
  so for small enough corner chops $\pack_\toric(\De_\varep)$
  is bounded away from $\pack_\toric(\De)$.
  Thus any set on which $\pack_\toric$ is continuous around $\De$ 
  cannot include any corner chops of $\De$.
  From this we conclude that any such set cannot include polytopes
  with greater than $N$ vertices> The result follows since
  is continuous on all of $\P_\toric^N$.
  \end{proof}

  Theorem~\ref{thmintro} part~\ref{thmintro_part_toric} follows from Theorem~\ref{thm_toricgeomandcombo} and Theorem~\ref{thm_toricpolycont}.
 In addition,  these Theorems also imply the following result.  Let  $N\geq 1$ and let $\sympplain^{2n,\T^n}_{\mathrm{T}, N}$ denote the 
  set of symplectic toric manifolds with exactly $N$ points
  fixed by the $\T^n$\--action. For $(M, \om, \phi)\in\sympplain^{2n,\T^n}_{\mathrm{T}, N}$ with fixed
    points $p_1, \ldots, p_N\in M$ let 
    $$
     \toricpack^i (M) = \left(\frac{\sup\{\,\mathrm{vol}(P)\mid P\textrm{ is a toric ball packing of $M$ such that $p_i\notin P$} \,\}}{\vol (\mathrm{B}^{2n})}\right)^{\frac{1}{2n}}.
    $$

 \begin{prop}\label{prop_toriccontnbhd}
     The space $\sympplain^{2n,\T^n}_{\mathrm{T}, N}$ is the largest
    neighborhood of $M$ in $\sympT$ in which $\toricpack$ is continuous if and
    only if 
    $ \toricpack^i (M) < \toricpack (M)$
    for every $1\leq i \leq N$.
 \end{prop}

Theorem~\ref{thmintro} part~\ref{thmintro_part_toric} and Proposition~\ref{prop_toriccontnbhd} are
illustrated in Figure~\ref{fig_toricfamilies}.
If $n=2$ Proposition~\ref{prop_toriccontnbhd} was proved in~\cite{FiPe2014}.

 \begin{figure}[ht]
  \centering
  \includegraphics[height=120pt]{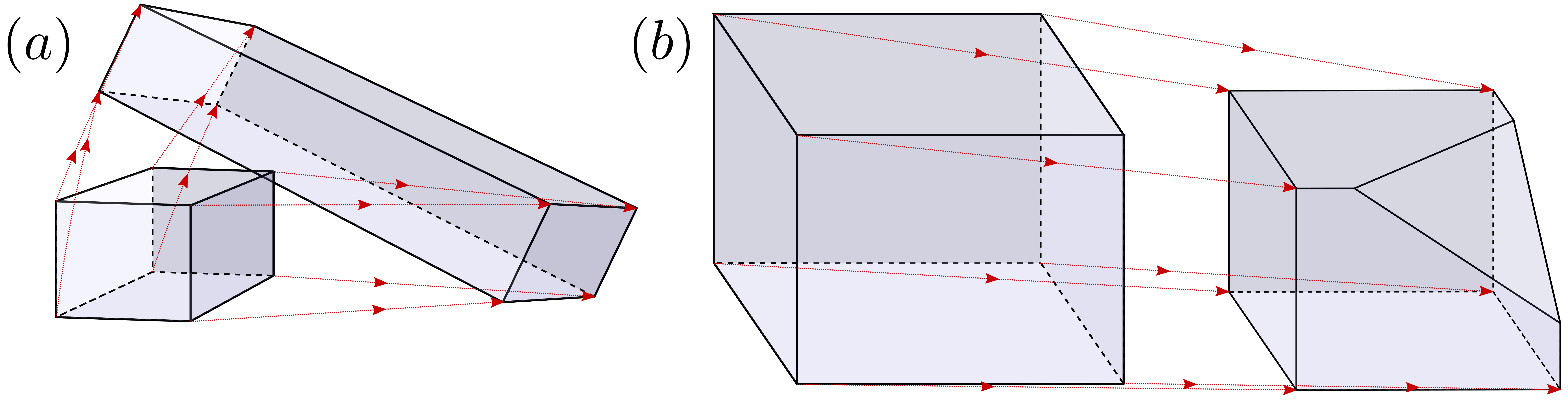}
  \caption{Continuous families of Delzant polygons on which $(a)$
  $\toricpack$ is continuous and $(b)$ $\toricpack$ is not continuous.}
  \label{fig_toricfamilies}
 \end{figure}

 \section{Continuity of symplectic \tex{$(S^1\times \R)$}\--capacities}
 \label{sec_stmetric}
 
 In this section we study the continuity of the symplectic 
 $(S^1\times \R)$\--capacity constructed in Section~\ref{sec_symplSRcap}.
 In~\cite{PaSTMetric2015} the second author defines a metric space structure
 on the moduli space of simple semitoric systems and in this section
 we will review this structure. We are only interested in the topology of $\msympST$ so, as is
 suggested in~\cite[Remark 1.31(3)]{PaSTMetric2015}, we will use a simplified
 version of the metric.  It is shown that while the simplified version
 produces a different metric space structure on $\msympST$ it induces the same 
 topology on $\msympST$ as the full 
 metric~\cite[Section 2.6]{PaSTMetric2015}. 
 
 Let us recall how the metric is constructed, since it is essential for the proofs of
 the upcoming results. One has a metric for every invariant (Definition~\ref{stlist}) and then~\cite{PaSTMetric2015}
 constructs a ``joint" metric from these. The first metric
 is the one on the Taylor series invariants, which is given as follows.
  A sequence $\{b_n\}_{n=0}^\infty$ with $b_n\in(0, \infty)$ is said to be 
  \emph{linear summable} if $\sum_{n=0}^\infty nb_n < \infty$. Let $\{b_n\}$
  be any such sequence and define
  $
   d_{\Rxyz}^{\{b_n\}_{n=0}^\infty} \left( (S^1)^\infty, (S^2)^\infty  \right) 
   $ to be
   $$
   \smashoperator{\sum_{i, j\geq0, (i,j)\neq(0,1)}} \min\big(\abs{\si^1_{i,j}-\si^2_{i,j}}, b_{i+j}\big)
                                                          + \min\big(\abs{\si^1_{0,1}-\si^2_{0,1}},2\pi-\abs{\si^1_{0,1}-\si^2_{0,1}},b_1\big)$$
  where 
  $(S^\ell)^\infty = \sum_{i, j\geq0}\si^\ell_{i,j}X^i Y^j\in\Rxyz$
  for $\ell = 1, 2$. 

  We denote the Lebesgue measure by $\la$
  and use $\symdiff$ to denote the symmetric difference.
  A measure $\nu$ on $\R^2$ is \emph{admissible} if it is in the same measure class as $\la$ (i.e. $\nu\ll\la$
  and $\la\ll\nu$) and  there exists some $g\colon \R\to\R$ such that
  the Radon-Nikodym derivative of $\nu$ with respect to $\la$ 
  satisfies
  $\frac{\mathrm{d}\nu}{\mathrm{d}\la}(x,y) = g(x)$
  for all $x,y\in\R$, where $g$ is bounded and bounded away from zero.
  
  Fix an admissible measure $\nu$.  For $\mf\in\Z_{\geq0}$ and $\vec{k}\in\Z^\mf$ let 
  $\pstpolygmfk$ denote the set of primitive semitoric
  polygons with complexity $\mf$ and twisting index $\vec{k}$ and let
  $\stpolygmfk$ denote the set of semitoric polygons which
  are the orbit of a primitive semitoric polygon in $\pstpolygmfk$.
  We may define
  $d_\P^\nu\colon \stpolygmfk\times\stpolygmfk\to[0, \infty)$
  by showing how it acts on orbits 
  $[\De_w^i]$
  elements    $\De_w^i = \big(\De^i, (\ell_{\la_j^i}, +1, k_j)_{j=1}^\mf\big) \in \pstpolygmfk$. If $\mf>0$,
  $$d_\P^\nu\big([\De_w^1],[\De_w^2]\big) = 
   \sum_{\vec{u}\in\{0,1\}^\mf} \nu \big( t_{\vec{\la}^1}^{\vec{u}}(\De^1)\symdiff 
                                          t_{\vec{\la}^2}^{\vec{u}}(\De^2)\big)$$
  and, if $\mf=0$,  $$d_\P^\nu\big([\De_w^1],[\De_w^2]\big) = \nu\big(\De^1\symdiff\De^2\big).$$
  For $I^i = \big( \mf, ((S_j^i)^\infty)_{j=1}^\mf, [\De_w^i], (h_j^i)_{j=1}^\mf\big)\in\ingred$,
  $i=1,2$ define
  $d_{\mf, \vec{k}}^{\nu, \{b_n\}_{n=0}^\infty}(I^1, I^2)$ to be
  \[
   d_\P^\nu ([\De_w^1], [\De_w^2]) + \sum_{j=1}^\mf \bigg( d_{\Rxyz}^{\{b_n\}_{n=0}^\infty}\big((S_j^1)^\infty,(S_j^2)^\infty\big)+\abs{h_j^1-h_j^2}\bigg)
  \]
  if $I^1, I^2\in\ingredmfk$ for some $\mf\in\Z_{\geq0},\vec{k}\in\Z^\mf$
  and otherwise define $d_{\mf, \vec{k}}^{\nu, \{b_n\}_{n=0}^\infty}(I^1, I^2)=1$.
  The \emph{metric $\mathcal{D}_{\rm ST}$ on $\msympST$} is the pullback of this one by  $\Phi$.
  It was shown in~\cite[Theorem A]{PaSTMetric2015} that the topology induced on
  $(\msympST,\mathcal{D}_{\rm ST})$  by the metric does not depend on
  the choice of $\nu$ or  $\{b_n\}_{n=0}^\mf$.

 Since $\mhamST$ is a quotient of $\hamST$ we can pull the topology up
 from $\mhamST$ to $\hamST$ by declaring that a set in $\hamST$ is open
 if and only if it is the preimage of an open set from $\mhamST$ under
 the natural projection.
 We endow $\sympST$ with the quotient topology relative to
 the map $\hamST\to\sympST$ which forgets the momentum map.
 Thus a map $c\colon \sympST\to [0,\infty]$ which descends to a well-defined map
 $\phi$ on $\msympST$ is continuous if and only if the map
 $\hat{c}\colon \mhamST\to [0,\infty]$ is continuous where $\hat{c}$ is
 defined by the commutative diagram:

  \begin{equation*}
   \begin{tikzcd}
   \hamST \ar{d}{}  \ar{r}[name = U]{} & \sympST \ar{d}{} \ar{r}{c} & {[0,\infty]} \\
   \mhamST \ar{r}{} \ar[bend right=55, swap]{rru}[name = D]{\hat{c}}          &  \msympST \ar{ur}{\phi}&   \\
   \end{tikzcd}
  \end{equation*}

 Let $\De_w = (\De, (\ell_{\la_j}, +1, k_j)_{j=1}^\mf)$ be
  a primitive semitoric polygon, and let $v\in\De$ be a vertex.
  \begin{definition}\label{def_admsemitoric}
   An \emph{admissible semitoric simplex of radius $r>0$ with
    center at $v$} is a subset $\Si$ of $\De$ such that 
    there exist some $A\in\mathrm{AGL}_2(\Z)$ and 
$\vec{u}\in\{0, 1\}^\mf$ satisfying:\\
    - $A(\De(r^{\nicefrac{1}{2}})) = t^{\vec{u}}_{\vec{\la}}(\Si)$;\\
    - $A(0) = t^{\vec{u}}_{\vec{\la}}(v)$;\\
    - $A$ takes the edges of $\De(r^{\nicefrac{1}{2}})$ meeting 
      at the origin to the edges 
      of $t^{\vec{u}}_{\vec{\la}}(\De)$ meeting at 
      $t^{\vec{u}}_{\vec{\la}}(v)$;\\
      -     $\Si \subset \De^{\vec{u}}$ where
     \[
      \De^{\vec{u}} = \De \setminus \left\{\, (x,y)\in\De \,\middle|\, \begin{array}{l}x = \la_j \textrm{ and }(-2\vec{u}+1)y\geq \min_{(\la_j, y_0)}y_0 + h_j\\ \textrm{ for some }j\in\{\,1, \ldots, \mf\,\}\end{array}\,\right\}.
     \]
    An \emph{admissible semitoric packing of $\De_w$} is a disjoint union
    $R = \bigsqcup_{\al\in \mathcal{A}} \Si_\al$
    where each $\Si_\al$ is an admissible simplex of some radius, where
    the radii of the simplices are allowed to be different.
  \end{definition}
   Such a simplex cannot exist at a fake corner.

 \begin{figure}[ht]
  \centering
  \includegraphics[height=85pt]{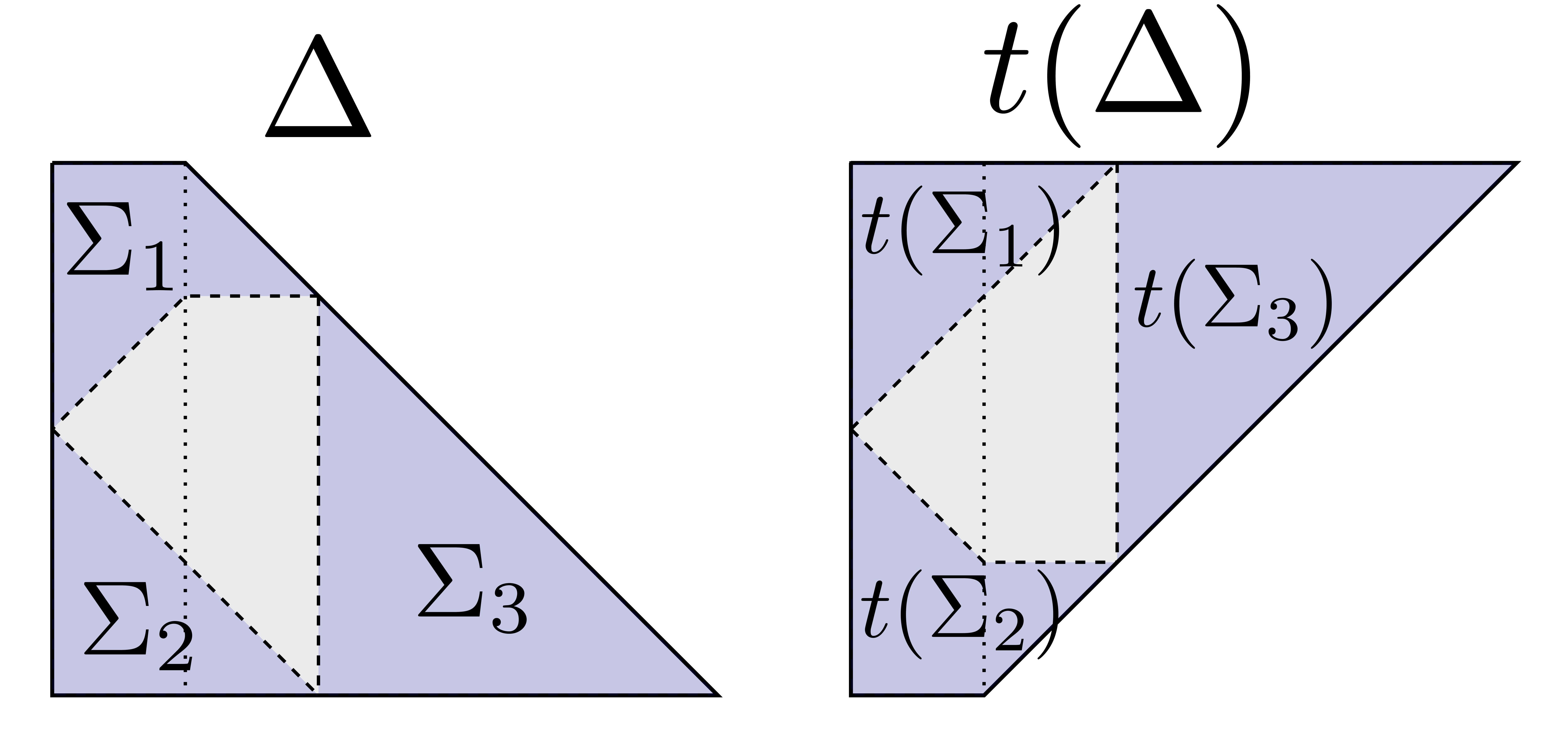}
  \caption{An admissible semitoric packing.  
  Here $t$ denotes $t_{\vec{\la}}^{\vec{u}}$.  }
  \label{fig_admstpacking}
 \end{figure}

 \begin{lemma}[\cite{Pe2007}]\label{lem_equivtoaffine}
  Let $F^B$ be a momentum map for the usual $\T^n$\--action
  on $\mathrm{B}^{2n}(r)$, $r>0$, and let $(M,\om,\phi, F)$ be a
  Hamiltonian $\T^n$\--manifold of dimension $2n$. If
  $\rho\colon \mathrm{B}^{2n}(r)\imm M$ is a 
  symplectic $\T^n$\--embedding with respect to some 
  $\La\in\mathrm{Aut}(\T^n)$
  then there exists some $x\in\R^n$ such that 
  the following diagram commutes:
  \begin{equation*}
   \begin{tikzcd}
   \mathrm{B}^{2n}(r) \ar{r}{\rho} \ar{d}[swap]{F^B} & M \ar{d}{F}\\
   \R^2 \ar{r}{(\La^t)^{-1}+x} & \R^2
   \end{tikzcd}
  \end{equation*}
  where 
  $(\La^t)^{-1}+x$ is the affine map with
  linear part $(\La^t)^{-1}$ which takes $0$ to $x$.
 \end{lemma}
 
 In~\cite{KaTo2005} a \emph{proper Hamiltonian $\T^n$\--manifold} is 
 a quadruple $(Q, \om^Q, F^Q, \Ga)$ where
 $(Q, \om^Q)$ is a connected 
 $2n$\--dimensional symplectic manifold with
 momentum map $F^Q$ for an action of $\T^n$
 and $\Ga \subset \mathrm{Lie}(\T^n)^*$ is
 an open convex subset with $F^Q(Q)\subset \Ga$
 and such that $F^Q$ is proper as a map to $\Ga$. 
 A proper Hamiltonian $\T^n$\--manifold is 
 \emph{centered about $p\in\Ga$} if 
 $p$ is an element of each component
 of $F^Q(Q^K)$ for each subgroup $K\subset\T^n$,
 where $Q^K$ is the set of all points in $Q$
 which are fixed by the action of all elements of $K$.
 
 \begin{lemma}[\cite{KaTo2005}]\label{lem_hamTmcenter}
  Let $(Q, \om^Q, F^Q, \Ga)$ be a proper Hamiltonian $\T^n$\--manifold
  of dimension $2n$. If $(Q, \om^Q, F^Q, \Ga)$ is centered about 
  $p\in\Ga$ and $(F^Q)^{-1}(\{p\})=\{q\}$, then $Q$ is equivariantly
  symplectomorphic to
  $
   \{\,z\in\C^n \mid p + \sum_{j=1}^n \abs{z_j}^2 \eta_j^q \in \Ga\,\},
  $
  where $\eta_1^q, \ldots, \eta_m^q\in\mathrm{Lie}(\T^n)^*$ are the weights of the isotropy
  representation of $\T^n$ on $T_q Q$.
 \end{lemma}

 We use Lemma~\ref{lem_equivtoaffine} and
 Lemma~\ref{lem_hamTmcenter} to prove the following.

 \begin{prop}\label{prop_stgeomandcombo}
  Let $(M, \om, F=(J,H))$ be a semitoric manifold such that
  \[ \Phi\big((M, \om, F)\big) = \big(\mf, ((S_j)^\infty)_{j=1}^\mf, [\De_w], (h_j)_{j=1}^\infty\big)\]
  where $\De_w = (\De, (\ell_{\la_j}, +1, k_j)_{j=1}^\mf)$ is primitive
  with associated momentum map $\widetilde{F} \in\mathcal{F}_M$
  such that $\widetilde{F}(M)=\De$.
  Then:
  \begin{enumerate}[font=\normalfont]
   \item\label{part_stcont1} Suppose $\rho\colon \mathrm{B}^{4}(r)\imm M$ is a semitoric
    embedding for some $r>0$.  Then $\widetilde{F}(\rho(\mathrm{B}^{4}(r)))\subset\De$ is an admissible
    semitoric simplex with radius $r^2$.  Conversely, if $\Si\subset\De$ is 
    an admissible semitoric simplex with radius $r^2$ then there exists a semitoric
    embedding $\rho\colon \mathrm{B}^{4}(r)\imm M$
    such that $\widetilde{F}(\rho(\mathrm{B}^{4}(r))) = \Si$.
   \item\label{part_stcont2} Let $P$ be a semitoric ball packing of $M$. Then
    $\widetilde{F}(P)\subset\De$ is an admissible packing of $\De_w$. Conversely,
    if $R$ is an admissible packing of $\De_w$ then there exists a
    semitoric ball packing $P$ of $M$ such that 
    $\widetilde{F}(P) = R$.
  \end{enumerate}
 \end{prop}
  
 \begin{proof}
  Part \eqref{part_stcont2} follows from Part 
  \eqref{part_stcont1} since the semitoric simplices
  associated to disjoint semitoricly embedded balls  
  are disjoint.  This follows from the fact that $\widetilde{F}^{-1}(p)$
  is a $2$\--dimensional submanifold of $M$ for any 
  regular point $p\in\De$ and the embedded balls 
  are $2$\--dimensional.
  
  Suppose that $B\subset M$ is a semitoricly embedded ball
  of radius $r>0$.  Then for some $\vec{\ep}\in\{-1,+1\}^\mf$
  the map $\rho_{\vec{\ep}}\colon \mathrm{B}^{4}(r)\imm M^{\vec{\ep}}$ is a 
  $\T^2$\--embedding with respect to some
  $\La\in\mathrm{Aut}(\T^2)$.
  Recall $M^{\vec{\ep}}$ is a Hamiltonian $\T^2$\--manifold
  and denote a momentum map for this action by $F^{\vec{\ep}}$.
  Let $p=F^{\vec{\ep}}(\rho(0))$ and let 
  $\De^{\vec{\ep}} = F^{\vec{\ep}}(M^{\vec{\ep}})$.
  Hence by Lemma~\ref{lem_equivtoaffine}
  the diagram
  \begin{equation*}
   \begin{tikzcd}
   \mathrm{B}^{4}(r) \ar{r}{\rho} \ar{d}[swap]{F_\mathrm{B}} & M^{\vec{\ep}} \ar{d}{F^{\vec{\ep}}}\\
   \De_\mathrm{B} \ar{r}{(\La^t)^{-1}+x} & \De^{\vec{\ep}}
   \end{tikzcd}
  \end{equation*}
  commutes for some $x\in\mathrm{Lie}(\T^2)^*$.  Since $\La$ is an automorphism so is $(\La^t)^{-1}$, hence it
  sends the weights of the isotropy representation of $\T^2$
  on $T_0 (\mathrm{B}^{4}(r))$ to the weights
  of the isotropy representation on $T_p M$.  Since $(\La^t)^{-1}$ is linear and $\De_\mathrm{B}$
  is the convex hull of the isotropy weights of the representation on $T_0 (\mathrm{B}^{4}(r))$ and the origin,
  we find that 
  \[\Si^{\vec{\ep}} := [(\La^t)^{-1}+x](\De_\mathrm{B})\]
  is the convex hull of $p$, $p+r^2 \al_1$, and $p+ r^2 \al_2$, minus
  the convex hull of $p+r^2\al_1$ and $p+r^2\al_2$,
  where $\al_1$ and $\al_2$ are
  the weights of the isotropy representation of $\T^2$ on $T_p M$.  For
  $\vec{u} = \frac{1}{2}(1-\vec{\ep})$ recall that 
  $t^{\vec{u}}_{\vec{\la}}(\De) = \De^{\vec{\ep}}$ and let 
  $
   \Si = \bigl(t^{\vec{u}}_{\vec{\la}}\bigr)^{-1}(\Si^{\vec{\ep}}).
  $
  Notice that $\Si = \widetilde{F}(\rho(\mathrm{B}^{4}(r)))\subset \De$ and is an admissible semitoric simplex.

  To prove the converse let $\Si\subset\De$ be an admissible semitoric simplex.
  This means that there exists some $\vec{\ep}\in\{-1,+1\}^\mf$ such that
  \[
   \Si' := t^{\vec{u}}_{\vec{\la}}(\Si)
  \]
  satisfies the requirements of
  Definition~\ref{def_admsemitoric}, where $\vec{u} = \frac{1}{2}(1-\vec{\ep})$.
  Let $\De' = t^{\vec{u}}_{\vec{\la}}(\De)$.  Let $p$ be the unique vertex of $\Si'$.
  Thus, $\Si'$ is the convex hull of $p$, $p + r^2 \al_1$, and $p+r^2 \al_2$,
  minus the convex hull of $p+r^2\al_1$ and $p+r^2\al_2$, for 
  some $\al_i\in\R^2$, $i=1,2$.   Let $\Ga\subset\R^2$ be the unique open half
  plane satisfying $\Ga\cup\De' = \Si'$.  Let $N=\widetilde{F}^{-1}(\Si)$
  and let $\om^N = \om|_{N}$.  We can see that $N\subset M$ is open
  and by the proof of the Atiyah-Guillemin-Sternberg Convexity 
  Theorem~\cite{At1982, GuSt1982} we know that $N$ is connected.
  The map $\widetilde{F}$ is proper
  because its first component, $J$,  is proper and thus
  $
   \widetilde{F}^N := t^{\vec{u}}_{\vec{\la}}\left( \widetilde{F}|_N\right)\colon  N \to \Si'
  $
  is proper.
  Therefore $\widetilde{F}^N\colon  N\to\Ga$
  is proper because $(\widetilde{F}^N)^{-1}(\Ga\setminus\Si')=\varnothing$,
  and hence $(N,\om^N, \widetilde{F}^N, \Ga)$
  is a proper Hamiltonian $\T^2$\--manifold.  
  Since $(N,\om^N, \widetilde{F}^N, \Ga)$ is centered about $p\in\R^2$
  by Lemma~\ref{lem_hamTmcenter} we conclude that $N$ is equivariantly
  symplectomorphic to
  \[
   \{\,z\in\C^2 \mid p + \abs{z_1}^2 \al_1 + \abs{z_2}^2 \al_2 \in \Ga\,\} = \mathrm{B}^4(r).
  \]
  It follows that there exists a symplectic $\T^2$\--embedding
  $\rho\colon \mathrm{B}^4(r)\imm M^{\vec{\ep}}$ with image $N$ so
  $
   \widetilde{F}(\rho(\mathrm{B}^4(r)))=\widetilde{F}(N)=\Si
  $.
 \end{proof}
 
 Define the \emph{optimal semitoric polygon packing function}
  $\pack_{\semitoric}\colon \stpolyg\to[0,\infty]$
  by
  \[\pack_{\semitoric}([\De_w]) = \sup\{\,\mathrm{vol}_\P (P)\mid P\textrm{ is an admissible semitoric packing of }\De_w\,\}. \]
  It is well-defined because any two primitive semitoric polygons in
  the same orbit are related to one another by a transformation 
  in $G_\mf\times\mathcal{G}$ which sends semitoric packings
  to semitoric packings and preserves volume.
 
 \begin{definition}
  We call $\al\in(0,\pi)$ a \emph{smooth angle} if it can be obtained
  as an angle in a Delzant polygon.  
 \end{definition}
  Equivalently, $\al\in(0,\pi)$ is
  smooth if and only if it is the angle at the origin of
  $A_{\al}(\De(1))$ for some $A_{\al}\in\mathrm{SL}_2(\Z)$.
 
 \begin{lemma}\label{lem_smoothanglesdiscrete}
  The set of smooth angles is discrete in $(0,\pi)\subset\R$.
 \end{lemma}

 \begin{proof}
  Fix a smooth angle $\al\in(0,\pi)$ and 
  fix some $\varep>0$ small enough so that
  $(\al-\varep, \al + \varep)\subset (0,\pi)$. Let 
  \[
   B_{\varep}(\al) = \{\,\be\in(0,\pi)\mid\be\textrm{ is a smooth angle and }\abs{\al-\be}<\varep\,\}
  \]
  and let $\de_{\varep}>0$ be such that if $\be\in B_{\varep}(\al)$ then
  $\abs{\sin (\al) - \sin (\be)}<\de_{\varep}$.
  Now fix any $\be\in B_{\varep}(\al)$.  This means there exists
  some $A_\be\in\mathrm{SL}_2(\Z)$ such that $\be$ is the angle at the
  origin of $\De=A_\be(\De(1))$.  Let $\ell_1, \ell_2\in\R$ denote the
  lengths of two edges of the simplex $\De$ which are adjacent to the
  vertex at the origin.  These each represent the magnitude of
  a vector in $\Z^n$ so $\ell_i\geq 1$ for $i=1, 2$.  
  By the choice of $\de_\varep$ we have that
  $\sin(\be) > \sin(\al) -\de_\varep$. Since $\De$ has area $\nicefrac{1}{2}$ we know that
  $\frac{\ell_1 \ell_2 \sin (\be)}{2} = \frac{1}{2}$
  and so for $i = 1, 2$ we conclude that
  $1 = \ell_1 \ell_2 \sin (\be) \geq \ell_i \sin (\be)$
  which implies that
  \[
   \ell_i \leq \frac{1}{\sin(\be)} < \frac{1}{\sin (\al) - \de_\varep}.
  \]
  Therefore associated to each
  $\be\in B_{\varep}(\al)$ there is a pair of vectors in $\Z^2$
  each with length less than $(\sin (\al) - \de_\varep)^{-1}$,
  a value which does not depend on $\be$.  There
  are only finitely many such vectors.
 \end{proof}
 The proof of Lemma~\ref{lem_smoothanglesdiscrete} is taken from the proof 
 of~\cite[Theorem 7.1]{FiPe2014} and is a two-dimensional version of 
 the strategy used in Theorem~\ref{thm_toricpolycont}. 
 Let $\al\in(0,\pi)$ be called a \emph{hidden smooth angle} if it can be
 obtained as a hidden corner in a primitive semitoric polygon.
  
 \begin{cor}\label{cor_smoothhiddendiscrete}
  The set of hidden smooth angles is discrete in $(0,\pi)\subset\R$.
 \end{cor}

 It is important to notice that a sequence of smooth angles can
 approach $\pi$. This must be the case, for example, if a semitoric
 polygon has infinitely many vertices.
 
 \begin{definition}\label{def_Nvert}
  We say that a vertex $v$ of
  $(\De, (\ell_{\la_j}, +1, k_j)_{j=1}^\mf))$
  is \emph{non-fake} if
  it is either Delzant or hidden in one, and hence all, elements
  of the affine invariant.
  For $N\geq1$ let $\pstpolygn$
  denote the set of primitive polygons with exactly $N$ non-fake
  vertices and let $\stpolygn$ denote the set of 
  $(G_\mf\times\mathcal{G})$\--orbits of elements of $\pstpolygn$. Let
  $\ingred^N$ be the set of all semitoric ingredients
  for which the affine invariant is an element of
  $\stpolygn$ and let 
  \[\sympplain^{4,S^1\times\R}_{\mathrm{ST}, N} = \Phi^{-1}(\ingred^N)\]
  where $\Phi$ is as in Equation~\eqref{eqn_Phi}.
 \end{definition}

 Recall $\mathcal{H}_p^{\varep}(v)$ defined in Equation
 \eqref{eqn_halfplane}.
 The following are two operations which can be
 performed on $[\De_w]$ to produce a new element of $\pstpolyg$.
 
 \begin{definition}\label{def_stchop}
 Let $\De_w =(\De, (\ell_{\la_j}, +1, k_j)_{j=1}^\mf)$.
 Let $p\in\De$ be a vertex and let 
 $v_1, v_2\in\Z^2$ be the primitive inwards pointing normal vectors
 to the two edges which meet at $p$ ordered so that
 $\det(v_1, v_2)>0$.

  If $p$ is a Delzant vertex of $\De_w$ then the
   \emph{$\varep$\--corner chop of $\De_w$ at $p$} is the 
   primitive semitoric polygon
   \[
    \De_w^{p, \varep} =\left(\De \cap \mathcal{H}_{p}^{\varep} (v_1+v_2), (\ell_{\la_j}, +1, k_j)_{j=1}^\mf\right).
   \]
   Similarly, given $[\De_w]$ we say that 
   $[\De_w^{p, \varep}]$ is the
   \emph{$\varep$\--corner chop of $[\De_w]$ at $p$}.
   
  Suppose $p$ is a hidden corner of $\De_w$ and thus
   there exists 
   $j\in\{\,1, \ldots, \mf\,\}$ such that $p\in\ell_{\la_j}$.  The
   \emph{$\varep$\--hidden corner chop of $\De_w$ at $p$} is 
   the primitive semitoric polygon
   \[
    \De_w^{p, \varep} =\left(\De \cap t_{\ell_{\la_j}}^{-1}\big(\mathcal{H}_{p}^{\varep} (v_1+v_2)\big), (\ell_{\la_j}, +1, k_j)_{j=1}^\mf\right).
   \]
   We say that 
   $[\De_w^{p, \varep}]$ is the
   \emph{$\varep$\--hidden corner chop of $[\De_w]$ at $p$}.
 \end{definition}
 
 The hidden corner chop of a hidden corner amounts to acting on the polygon
 with $t_{\ell_{\la_j}}^1$ to transform the hidden corner into a Delzant
 corner, performing the usual corner chop on this Delzant corner, and then
 transforming the polygon back with $t_{\ell_{\la_j}}^{-1}$.  This is
 shown in Figure~\ref{fig_hiddencornerchop}.

 \begin{figure}[ht]
 \centering
  \includegraphics[height=65pt]{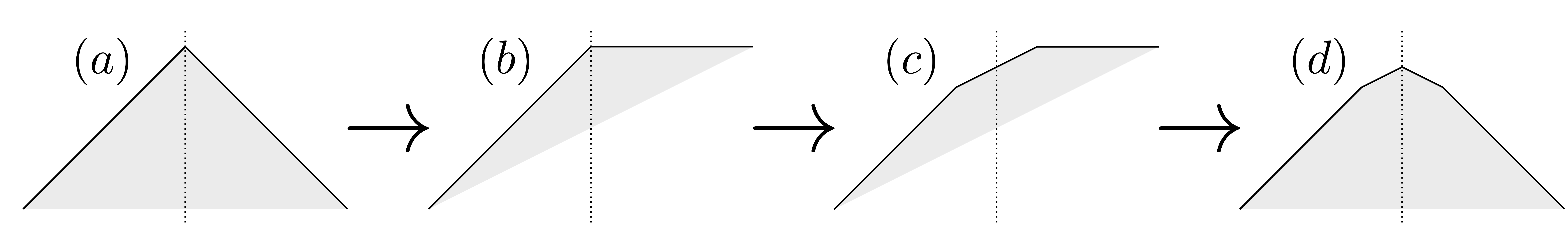}
  \caption{
  In $(a)$ a hidden corner
  is shown.  In $(b)$ we unfold it by reversing the sign of the
  associated $\ep_i$ resulting in a Delzant corner.  In $(c)$ we perform
   corner chop on this corner and in $(d)$ the 
  $\ep_i$ returns to its original sign.}
  \label{fig_hiddencornerchop}
 \end{figure}
  
 \begin{lemma}\label{lem_stsimpleobs}
  Fix $N\in\Z_{\geq0}$. 
  Each $[\De_w]\in\stpolygn$
  has an open neighborhood
  in $\stpolygn$ which consists
  exclusively of transformations of $[\De_w]$ in which its sides
  are moved in a parallel way.
  Moreover, any sufficiently small neighborhood of $[\De_w]$
  in $\stpolyg$
  is contained in $\cup_{(N'\geq N)}\stpolygnprime$.  
 \end{lemma}
 
 \begin{proof}
  The angles of non-fake corners are discrete by
  Lemma~\ref{lem_smoothanglesdiscrete} and
  Corollary~\ref{cor_smoothhiddendiscrete}.
  This means that there exists a neighborhood of $[\De_w]$
  in which all elements which have $N$ non-fake vertices
  must have all of the same angles as $[\De_w]$.  This is
  the open neighborhood described in the Lemma.
  Any semitoric polygon
  with fewer non-fake vertices than $[\De_w]$ is bounded
  away from $[\De_w]$ because the only ways to change the number
  of non-fake vertices are a corner chop or introducing a smooth
  angle into an edge of infinite length but by
  Lemma~\ref{lem_smoothanglesdiscrete} smooth angles are discrete.
 \end{proof}

 \begin{lemma}\label{lem_stpackdiscont}
  The map $\pack_\semitoric\colon \stpolyg\to[0,\infty]$ is discontinuous at every point.
 \end{lemma} 
 
 \begin{proof}
  Primitive semitoric polygons must have
  at least one non-fake vertex.  Let 
  \[[\De_w] = [(\De, (\ell_{\la_j}, +1, k_j)_{j=1}^\mf)]\]
  be a semitoric polygon.
  First assume that $[\De_w]\in\stpolygn$ for some
  $N\geq 1$ and that $\pack_\semitoric([\De_w])<\infty$.
  Then for $\varep>0$ small enough define $[\De_w^\varep]$ to be
  the semitoric polygon produced by performing an 
  $\varep$\--corner chop at each
  non-fake vertex of $[\De_w]$. We have that
  \begin{equation}\label{eqn_stpackdisconteqn}
  \lim_{\varep\to0}d^\P_\semitoric ([\De], [\De_w^\varep]) = 0.
  \end{equation}
  A packing of $[\De_w^\varep]$ has
  at most $2N$ disjoint admissible simplices.  Since their side lengths
  are determined by the lengths of the adjacent
  edges, one of which is length $\varep$, we
  have that
  $\lim_{\varep\to 0}\pack_\semitoric ([\De_w^\varep]) = 0$.
  Since every semitoric polygon has positive optimal packing
  we have
  \[
   \lim_{\varep\to 0} \abs{\pack_\semitoric([\De_w])-\pack_\semitoric(\De_w^\varep)} = \pack_\semitoric([\De_w])>0
  \]
  and thus, in light of Equation \eqref{eqn_stpackdisconteqn},
  $\pack_\semitoric$ is discontinuous at $[\De_w]$.
  
  Suppose $[\De_w]\in\stpolygn$ for 
  some $N\geq 1$ and 
  $\pack_\semitoric([\De_w])=\infty$.
  Since $[\De_w]$ has only finitely many non-fake vertices, any admissible
  packing has only finitely many admissible simplices.
  Hence there is a vertex
  at which an arbitrarily large simplex fits. The only possible case is
  that $N=1$ and the polygon is of complexity zero. 
  Taking a corner chop of any size at the single non-fake vertex
  produces a polygon on which $\pack_\semitoric$ evaluates to a finite number, so
  $\pack_\semitoric$ is discontinuous at $[\De_w]$.
  
  Now suppose that $\pack_\semitoric ([\De_w])<\infty$ and
  $[\De_w]\in\stpolyg\setminus\bigcup_{N\geq1}\stpolygn$.
  For $i \in\Z_{\geq1}$ let $I_i \subset\R$ be
  given by
  $I_i = [-n,n]\setminus(-(n-1), n-1)$
  and let $N_i\in\Z_{\geq0}$ denote the number of 
  non-fake vertices of $[\De_w]$ 
  with $x$\--coordinate in
  $I_i$. This number is finite by the definition
  of a convex polygon and it is invariant under the action of
  $G_\mf\times\mathcal{G}$.  For $\varep>0$ small enough
  let $[\De_w^\varep]$ be a semitoric polygon which has a
  small corner chop at each non-fake vertex such that, at each vertex
  in $I_i$ for $i\in\Z_{\geq1}$,
  the largest possible admissible simplex that can 
  fit into that vertex has volume at most
  $\nicefrac{\varep}{(N_i S^{i+1})}$.  Then an admissible packing
  $R$ of $[\De_w^\varep]$ satisfies
  \[\mathrm{vol}_\P (R)\leq \sum_{i=1}^\infty \frac{\varep}{N_i 2^{i+1}}2N_i = \varep.\]
  Therefore
 $$\lim_{\varep\to0}d^\P_\semitoric ([\De_w],[\De_w^\varep]) = 0$$
  while
  $$\lim_{\varep\to0} \abs{\pack_\semitoric ([\De_w]) - \pack_\semitoric ([\De_w^\varep])} = \pack_\semitoric ([\De_w]) > 0$$
  and thus $\pack_\semitoric$ is not continuous at $[\De_w]$.
 \end{proof}
 
 For $[\De_w]=[(\De, (\ell_{\la_j}, +1, k_j)_{j=1}^\mf)] \in \stpolygn$
    with non-fake vertices $v_1, \ldots, v_N$,
    let $\pack_{\semitoric}^{\P,i} (\De)$ be
    the total volume of the optimal packing
    excluding all packings which have a simplex centered at $v_i$.
 \begin{theorem}\label{thm_stpolycont}
  Let $\pack_{\semitoric}\colon \stpolyg\to[0, \infty]$
  be the optimal semitoric polygon packing function. Then:
  \begin{enumerate}[font=\normalfont]
   \item\label{part_stpolythm1} $\pack_{\semitoric}$ is discontinuous at 
    each point in $\stpolyg$;
   \item\label{part_stpolythm2} the restriction 
    $\pack_{\semitoric}|_{\stpolygn}$ is continuous for each $N\in\Z_{\geq1}$;
   \item\label{part_stpolythm3}
    if $[\De_w] \in \stpolygn$
    then $\stpolygn$ is the largest
    neighborhood of $\De_w$ in $\stpolygn$ in which 
    $\pack_{\semitoric}$ is continuous if and only if 
    $\pack_{\semitoric}^{i} ([\De_w]) < \pack_{\semitoric} ([\De_w])$
    for all $1\leq i \leq N$.
  \end{enumerate}
 \end{theorem}

 \begin{proof}
  Part \eqref{part_stpolythm1} is the content of Lemma~\ref{lem_stpackdiscont}.
  
  By Lemma~\ref{lem_stsimpleobs}, given any
  $[\De_w]\in\stpolygn$,
  there exists a neighborhood of $[\De_w]$ in $\stpolygn$ containing exclusively
  orbits of polygons formed by translating the sides of $\De_w$ in a parallel way.
  Hence part~\eqref{part_stpolythm2} follows from this because $\pack_\semitoric$ is continuous
  on such transformations. 
  
  For Part \eqref{part_stpolythm3} suppose first that
  $\pack_{\semitoric} ([\De_w]) = \pack_{\semitoric}^{i} ([\De_w])$
  for some $i\in\{\,1, \ldots, N\,\}$. This means that there exists some
  optimal packing avoiding the $i^{\mathrm{th}}$ non-fake vertex. For
  $\varep>0$ let $[\De_w^\varep]$ be the result of an 
  $\varep$\--corner chop at the $i^{\mathrm{th}}$ vertex and
  notice that
  $\lim_{\varep\to0}d^\P_\semitoric ([\De_w], [\De_w^\varep])= 0$ and  
  $\lim_{\varep\to0} \pack_\semitoric ([\De_w^\varep]) = \pack_\semitoric ([\De_w])$.
Thus there exists some set larger
  than $\stpolygn$ on which 
  $\pack_\semitoric$ is continuous, as 
  shown
  in Figure~\ref{fig_conthiddencornerchop}.
  \begin{figure}[ht]
  \centering
   \includegraphics[height=60pt]{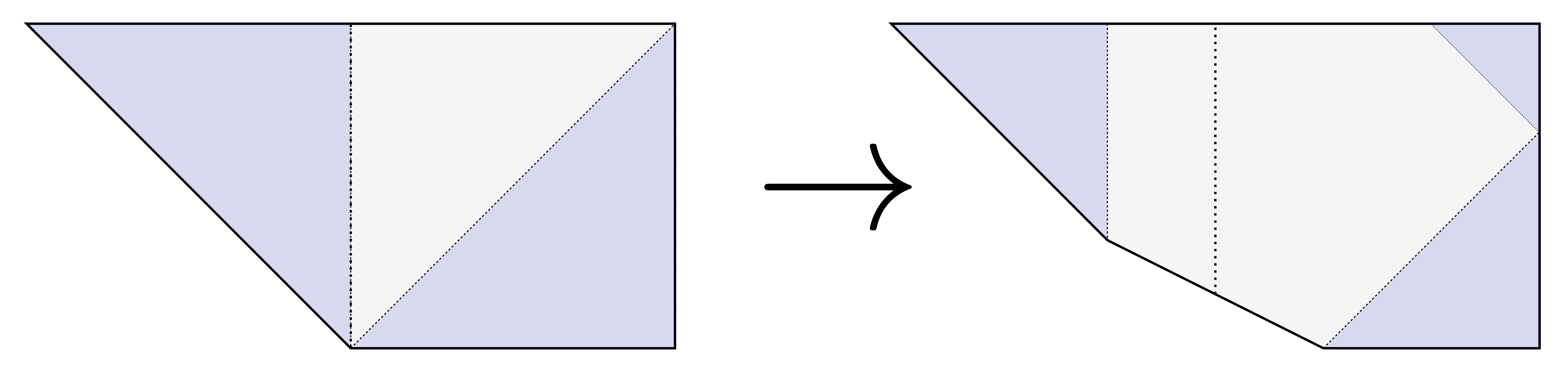}
   \caption{Corner chop of a corner not used in the optimal packing.}
   \label{fig_conthiddencornerchop}
  \end{figure}
  
  Finally, to show the converse assume that $[\De_w]$ satisfies
  $\pack_{\semitoric}^{\P,i} ([\De_w]) < \pack_{\semitoric}^\P ([\De_w])$
  for all $1\leq i \leq N$.
  By Lemma~\ref{lem_stsimpleobs} there
  is an open set around $[\De_w]$ in which the only elements not in 
  $\pstpolygn$ are obtained from $[\De_w]$
  by iterations of corner chops,
  parallel translations of the edges, and
  introducing a smooth angle into an edge of infinite length.  For $\varep>0$
  let $[\De_w^\varep]$ be any $\varep$\--corner chop at the 
  $i^{\mathrm{th}}$ non-fake vertex of $[\De_w]$.  Then
  \[\lim_{\varep\to0}\pack_\semitoric ([\De_w^\varep]) = \pack_{\semitoric}^{i} ([\De_w]) < \pack_{\semitoric} ([\De_w])\]
  and the result follows.
  \end{proof}

 Notice that the quotient map $\sympST\to\stpolyg$ is continuous and the
 metric on $\sympST$ is the sum of the metric on $\stpolyg$ and the 
 metric on the remaining components.  Thus,
 Theorem~\ref{thmintro} part~\ref{thmintro_part_semitoric} follows from Theorem~\ref{thm_stpolycont}. 
 For $(M, \om, F)\in\sympplain^{4,S^1\times\R}_{\mathrm{ST}, N}$ 
 with fixed points $p_1, \ldots, p_N\in M$ let
 \[
  \semitoricpack^i(M) = \left( \frac{\sup\{\,\mathrm{vol}(P)\mid P\subset M\textrm{ is a semitoric ball packing of $M$ and }p_i\notin P \,\}}{\vol (\mathrm{B}^{4})}\right)^{\frac{1}{4}}.
 \]

 \begin{prop}\label{prop_stcontnbhd}
  Let $N\geq 1$.  
  If $(M, \om, F)\in\sympplain^{4, S^1\times\R}_{\mathrm{ST}, N}$ 
  then $\sympplain^{4, S^1\times\R}_{\mathrm{ST}, N}$ is the largest
  neighborhood of $M$ in $\sympST$ in which $\semitoricpack$ 
  is continuous if and only if 
  $ \semitoricpack^i (M) < \semitoricpack (M)$
  for all $1\leq i \leq N$.
 \end{prop}

 Theorem~\ref{thmintro} part~\ref{thmintro_part_semitoric} and Proposition~\ref{prop_stcontnbhd} are illustrated
 in Figure~\ref{fig_semitoricfamilies}.
 
 \begin{figure}[ht]
 \centering
  \includegraphics[height=95pt]{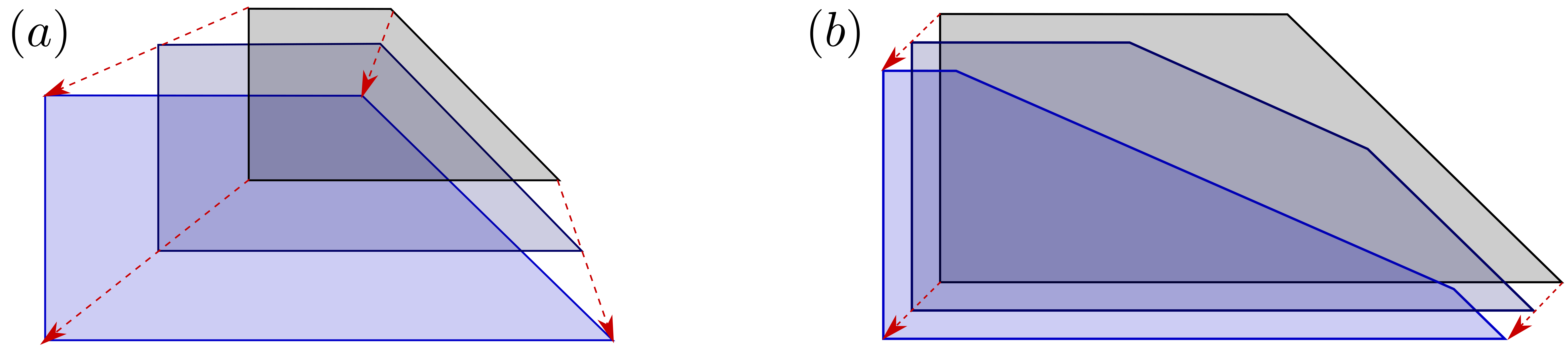}
  \caption{Continuous families of primitive semitoric polygons on which $(a)$
  $\semitoricpack$ is continuous and $(b)$ $\semitoricpack$ is not continuous.}
  \label{fig_semitoricfamilies}
 \end{figure}

 \begin{definition}
   The \emph{semitoric radius capacity} is the symplectic $(S^1\times\R)$\--capacity 
    $\semitoricpack_{\mathrm{rad}}\colon \sympST\to[0,\infty]$ given by
    $$
     \semitoricpack_{\mathrm{rad}}(M)=\mathrm{sup}\{\,r>0\mid\textrm{there exists a semitoric embedding } \mathrm{B}^4(r)\imm M\,\}.
    $$
  \end{definition}

 It can be shown that $\semitoricpack_{\mathrm{rad}}$ is a $(S^1\times\R)$\--capacity in the same
 way that it was shown that $\semitoricpack$ is a $(S^1\times\R)$\--capacity.
 Recall that $\sympplain^{2n,\R^n}_{\mathrm{T}}$ is the symplectic $\R^n$\--category
 which is the collection of toric manifolds with their $\T^n$\--action lifted to an
 $\R^n$\--action. Let $\sympplain^{2n,\R^n}_{\mathrm{T}, N}$ denote those systems with
 exactly $N$ points fixed by the $\R^n$\--action.
 By repeating the proofs of the continuity
 results Theorem~\ref{thmintro} part~\ref{thmintro_part_toric}, Proposition~\ref{prop_toriccontnbhd},
 Theorem~\ref{thmintro} part~\ref{thmintro_part_semitoric}, and Proposition~\ref{prop_stcontnbhd}
 we immediately have the following result, that yields Theorem~\ref{thmintro} part~\ref{thmintro_part_gromov}.

 \begin{theorem}
  The maps $\cBnn|_{\sympplain^{2n,\R^n}_{\mathrm{T}}}$ and $\semitoricpack_{\mathrm{rad}}$ are
  discontinuous everywhere on their domains and
  the restrictions $\cBnn|_{\sympplain^{2n,\R^n}_{\mathrm{T}, N}}$ and $\semitoricpack_{\mathrm{rad}}|_{\sympplain^{4, S^1\times\R}_{\mathrm{ST}, N}}$
  are both continuous. 
  For $(M, \om, F)\in\sympplain^{2n, \R^n}_{\mathrm{T},N}$ 
  the set $\sympplain^{2n,\R^n}_{\mathrm{T},N}$ is not the largest
  neighborhood of $M$ in $\sympplain^{2n,\R^n}_{\mathrm{T}}$ in which 
  $\cBnn|_{\sympplain^{2n,\R^n}_{\mathrm{T}}}$ is continuous and for
  $(M, \om, F)\in\sympplain^{4,S^1\times\R}_{\mathrm{ST}, N}$ 
  the set $\sympplain^{4,S^1\times\R}_{\mathrm{ST}, N}$ is the largest
  neighborhood of $M$ in $\sympST$ in which $\semitoricpack_{\mathrm{rad}}$ 
  is continuous if and only if $N=1$.
 \end{theorem}

\begin{remark}
There are many examples of classical symplectic capacities
(see for instance~\cite{CiHoLaSc2007}), and
it would be of interest to adapt these capacities to the equivariant
category. It would also be useful to construct symplectic $G$\--capacities for more general integrable systems.
In particular, integrable systems where a complete list 
of invariants is not known (that is, the vast majority). 

In~\cite{GoPeSa2016} the authors give a lower bound on the number of fixed points of 
a circle action on a 
compact almost complex  manifold $M$ with nonempty fixed point set, under the condition 
that the Chern number $c_1 c_{n-1}[M]$ vanishes.  These results apply to a class of manifolds 
which do not support any Hamiltonian circle action with isolated fixed points,
and which includes all symplectic
Calabi\--Yau manifolds~\cite{Yau2011} (see~\cite[Proposition 2.15]{GoPeSa2016}).
The class of symplectic Calabi\--Yau manifolds is thus of particular interest
because they do not admit integrable systems of toric or semitoric type.
Also, there is work extending the classification
in~\cite{PeVNconstruct2011} and related results to higher dimensions~\cite{Waaction2015},
so one could extend the semitoric packing capacity to higher dimensional semitoric systems, 
for which there is currently no classification.

Another interesting direction would be to generalize the work in~\cite{Ku1995} to our setting.
There, the author constructs infinite dimensional symplectic capacities for a general class of Hamiltonian PDEs.
In case the PDEs preserves some $G$\--action, one may expect to construct also $G$\--capacities in such infinite
dimensional setting, and this may give new interesting result on the long time behavior of solutions.

Symplectic capacities are also of interest from a physical view point,
for instance in~\cite{GoLu2009} the authors describe interrelations between symplectic capacities
and the uncertainty principle. It would be interesting to explore
similar connections to symplectic $G$\--capacities.
\end{remark}

\begin{remark}
In this paper $G$ can be a compact Lie group (like in the
case of symplectic toric manifolds)
or a non-compact Lie group (like in the case of semitoric
systems). In general there are obstructions to the existence of
effective $G$\--actions on compact and non-compact manifolds,
even in the case that the $G$\--action is only required
to be smooth. For instance, in~\cite[Corollary in page 242]{Yau1977} it is
proved that if $N$ is an $n$\--dimensional manifold on which
a compact connected Lie group $G$ acts effectively and there
are $\sigma_1,\ldots,\sigma_n \in {\rm H}^1(M,\mathbb{Q})$
such that $\sigma_1\cup \ldots \cup \sigma_n \neq 0$ then
$G$ is a torus and the $G$\--action is locally free. In~\cite{Yau1977}
Yau also proves several other results giving restrictions on
$G$, $M$, and the fixed point set $M^G$. If the $G$\--action
is moreover assumed to be symplectic or K\"ahler, there are
even more non-trivial constraints. Therefore the class of
symplectic manifolds for which one can define a notion of
symplectic $G$\--capacity with $G$ non-trivial 
is in general much more restrictive than the class of
all symplectic manifolds.
\end{remark}

\bibliographystyle{amsplain}
\bibliography{biblio}

\providecommand{\bysame}{\leavevmode\hbox to3em{\hrulefill}\thinspace}
\providecommand{\MR}{\relax\ifhmode\unskip\space\fi MR }
\providecommand{\MRhref}[2]{%
  \href{http://www.ams.org/mathscinet-getitem?mr=#1}{#2}
}
\providecommand{\href}[2]{#2}
\begin{thebibliography}{10}

\bibitem{At1982}
M.~Atiyah, \emph{Convexity and commuting {H}amiltonians}, Bull. London Math.
  Soc. \textbf{14} (1982), 1--15.

\bibitem{Ba1995}
S.~M. Bates, \emph{Some simple continuity properties of symplectic capacities},
  The {F}loer memorial volume, Progr. Math., vol. 133, Birkh\"auser, Basel,
  1995, pp.~185--193.

\bibitem{CiHoLaSc2007}
K.~Cieliebak, H.~Hofer, J.~Latschev, and F.~Schlenk, \emph{Quantitative
  symplectic geometry}, Dynamics, ergodic theory, and geometry, Math. Sci. Res.
  Inst. Publ., vol.~54, Cambridge Univ. Press, Cambridge, 2007, pp.~1--44.

\bibitem{GoLu2009}
M.~de~Gosson and F.~Luef, \emph{Symplectic capacities and the geometry of
  uncertainty: the irruption of symplectic topology in classical and quantum
  mechanics}, Phys. Rep. \textbf{484} (2009), no.~5, 131--179.

\bibitem{De1988}
T.~Delzant, \emph{{H}amiltoniens p\'{e}riodiques et image convex de
  l'application moment}, Bull. Soc. Math. France \textbf{116} (1988), 315--339.

\bibitem{EkHo1989}
I.~Ekeland and H.~Hofer, \emph{Symplectic topology and {H}amiltonian dynamics},
  Math. Z. \textbf{200} (1989), no.~3, 355--378.

\bibitem{FiPe2014}
A.~Figalli and {\'A}.~Pelayo, \emph{Continuity of ball packing density on
  moduli spaces of toric manifolds}, Advances in Geometry, (to appear)
  arXiv:1408.1462.

\bibitem{GoPeSa2016}
L.~Godinho, {\'{A}}.~Pelayo, and S.~Sabatini, \emph{Fermat and the number of
  fixed points of periodic flows}, arXiv:1404.4541.

\bibitem{Gr1985}
M.~Gromov, \emph{Pseudoholomorphic curves in symplectic manifolds}, Invent.
  Math. \textbf{82} (1985), no.~2, 307--347.

\bibitem{GuSt1982}
V.~Guillemin and S.~Sternberg, \emph{Convexity properties of the moment
  mapping}, Invent. Math. \textbf{67} (1982), 491--513.

\bibitem{Hocap1990}
H.~Hofer, \emph{Symplectic capacities}, Geometry of low-dimensional manifolds,
  2 ({D}urham, 1989), London Math. Soc. Lecture Note Ser., vol. 151, Cambridge
  Univ. Press, Cambridge, 1990, pp.~15--34.

\bibitem{KaPaPe2015}
D.~M. Kane, J.~Palmer, and \'A. Pelayo, \emph{Classifying toric and semitoric
  fans by lifting equations from $\rm{SL}_2(\mathbb{Z})$}, arXiv:1502.07698.

\bibitem{KaTo2005}
Y.~Karshon and S.~Tolman, \emph{The {G}romov width of complex {G}rassmannians},
  Alg. and Geo. Topology \textbf{5} (2005), 911--922.

\bibitem{Ku1995}
S.~Kuksin, \emph{Infinite-dimensional symplectic capacities and a squeezing
  theorem for {H}amiltonian {PDE}s}, Comm. Math. Phys. \textbf{167} (1995),
  no.~3, 531--552.

\bibitem{PaSTMetric2015}
J.~Palmer, \emph{Moduli spaces of semitoric systems}, arXiv:1502.07296.

\bibitem{Pe2006}
{\'A}~Pelayo, \emph{Toric symplectic ball packing}, Topology and its Appl.
  \textbf{157} (2006), 3633--3644.

\bibitem{Pe2007}
{\'A}.~Pelayo, \emph{Topology of spaces of equivariant symplectic embeddings},
  Proc. Amer. Math. Soc. \textbf{135} (2007), no.~1, 277--288.

\bibitem{PePRS2013}
{\'A}.~Pelayo, A.R. Pires, T.~Ratiu, and S.~Sabatini, \emph{Moduli spaces of
  toric manifolds}, Geometriae Dedicata \textbf{169} (2014), 323--341.

\bibitem{PeSc2008}
{\'A}~Pelayo and B.~Schmidt, \emph{Maximal ball packings of symplectic-toric
  manifolds}, Intern. Math. Res. Not. (2008), 24p, ID rnm139.

\bibitem{PeVNconstruct2011}
{\'A}.~Pelayo and S.~V{\~u}~Ng\d{o}c, \emph{Constructing integrable systems of
  semitoric type}, Acta Math. \textbf{206} (2011), 93--125.

\bibitem{PeVNsymplthy2011}
\bysame, \emph{Symplectic theory of completely integrable {H}amiltonian
  systems}, Bull. Amer. Math. Soc. \textbf{48} (2011), 409--455.

\bibitem{VN2003}
S.~V\~{u}~Ng\d{o}c, \emph{On semi-global invariants of focus-focus
  singularities}, Topology \textbf{42} (2003), no.~2, 365--380.

\bibitem{VN2007}
\bysame, \emph{Moment polytopes for symplectic manifolds with monodromy}, Adv.
  Math. \textbf{208} (2007), no.~2, 909--934.

\bibitem{Waaction2015}
C.~Wacheux, \emph{Asymptotics of action variables near semi-toric
  singularities}, J. of Geom. and Phys., (to appear) arXiv:1412.2414.

\bibitem{Yau1977}
S.~T. Yau, \emph{Remarks on the group of isometries of a {R}iemannian
  manifold}, Topology \textbf{16} (1977), no.~3, 239--247.

\bibitem{Yau2011}
\bysame, \emph{A survey of {C}alabi-{Y}au manifolds}, Geometry and analysis.
  {N}o. 2, Adv. Lect. Math. (ALM), vol.~18, Int. Press, Somerville, MA, 2011,
  pp.~521--563.

\bibitem{ZeZi2013}
K.~Zehmisch and F.~Ziltener, \emph{Discontinuous symplectic capacities}, J.
  Fixed Point Theory Appl. \textbf{14} (2013), no.~1, 299--307.

\end{thebibliography}

{\small
  \noindent
  \\
  {\bf Alessio Figalli} \\
  University of Texas at Austin\\
  Mathematics Dept. RLM 8.100 \\
  2515 Speedway Stop C1200\\
  Austin TX, 78712-1082, USA.\\
  {\em E\--mail}: \texttt{figalli@math.utexas.edu} \\
  \noindent
   \\
  {\bf Joseph Palmer} \\
  {\bf \'Alvaro Pelayo} \\
  University of California, San Diego\\ 
  Department of Mathematics\\
  9500 Gilman Drive \#0112\\
  La Jolla, CA 92093-0112, USA.\\
  {\em E\--mail}: \texttt{j5palmer@ucsd.edu} \\
  {\em E\--mail}: \texttt{alpelayo@ucsd.edu} \\
  \noindent
 }   

\end{document}